\documentclass[11pt,oneside,reqno]{amsart}
\usepackage[utf8]{inputenc}
\usepackage{amssymb}
\usepackage{amsmath}
\usepackage{amsthm}
\usepackage{mathtools}
\usepackage{enumitem}
\usepackage{caption}
\usepackage{hyperref}
\usepackage[a4paper]{geometry}

\newtheorem{theorem}{Theorem}[section]
\newtheorem*{theorem*}{Theorem}
\newtheorem{lemma}[theorem]{Lemma}
\newtheorem{prop}[theorem]{Proposition} 
\newtheorem{cor}[theorem]{Corollary}

\theoremstyle{definition}
\newtheorem{remark}[theorem]{Remark}

\newtheorem{defn}[theorem]{Definition}

\makeatletter
\renewcommand{\section}{\@startsection{section}{1}%
  \z@{2.5ex plus 1ex minus .2ex}{1.3ex plus .2ex}%
  {\normalfont\large\scshape\centering}}
 
\newcommand{\N}{\mathbb{N}}
\newcommand{\Z}{\mathbb{Z}}
\newcommand{\G}{\mathbb{G}}

\newcommand{\QQ}{\mathcal{Q}}
\newcommand{\R}{\mathbb{R}}
\renewcommand{\H}{\mathbb{H}}
\newcommand{\W}{\mathbb{W}}
\newcommand{\V}{\mathbb{V}}
\newcommand{\HH}{\mathcal{H}}
\renewcommand{\L}{\mathcal{L}}

\newcommand{\eps}{\varepsilon}

\DeclareMathOperator{\spann}{span}
\DeclareMathOperator{\diam}{diam}
\DeclareMathOperator{\graph}{graph}
\DeclareMathOperator{\Lie}{Lie}

\makeatletter
\def\l@subsection{\@tocline{2}{0pt}{2.5pc}{5pc}{}}
\makeatother

\title[Unrectifiable surfaces in Carnot groups]{Unrectifiable surfaces with intrinsic regularity in Carnot groups}

\author[V. Magnani]{Valentino Magnani}
\address[Valentino Magnani]{ 
Dipartimento di Matematica, 
Universit\`a di Pisa, 
Largo Bruno Pontecorvo 5, 56127 Pisa, Italy }
\email[Valentino Magnani]{valentino.magnani@unipi.it}

\author[J. Mirmina]{Jeremy Mirmina}
\address[Jeremy Mirmina]{ 
Dipartimento di Matematica, 
Universit\`a di Pisa, 
Largo Bruno Pontecorvo 5, 56127 Pisa, Italy }
\email[Jeremy Mirmina]{jeremy.mirmina@phd.unipi.it, jeremy.mirmina@unitn.it}

\date{}



\thanks{J.M. and V.M. were supported by the MIUR Excellence Department Project awarded to the Department of Mathematics, University of Pisa, CUP I57G22000700001. V.M. acknowledges the support of PRIN 2022PJ9EFL {\em Geometric Measure Theory: Structure of Singular Measures, Regularity Theory and Applications in the Calculus of Variations}, funded by the European Union--NextGenerationEU, CUP E53D23005860006}

\begin{document}

\begin{abstract}
In any noncommutative Carnot group, for every $0\le \alpha<1$, we construct a $C^{1,\alpha}_\mathrm H$ surface that is unrectifiable with respect to the Euclidean distance. Precisely, the surface has Euclidean Hausdorff dimension higher than its topological dimension. 
In step-2 Carnot groups, we show that every $C^{1,\alpha}_{\mathrm H}$ hypersurface has Euclidean Hausdorff dimension less than or equal to $q-\frac{1+\alpha}2$, where $q$ is the topological dimension of the group.
This upper bound is actually attained by the $C^{1,\alpha}_{\mathrm H}$ hypersurfaces of our construction, hence their Euclidean Hausdorff dimension is maximal.
As a technical result of independent interest, we show that the implicit function associated with a $C^{1,\alpha}_{\mathrm H}$ hypersurface is 
$\frac{1+\alpha}2$-Euclidean H\"older continuous, with $0\le\alpha\le 1$.
In particular, $C^{1,1}_{\mathrm H}$ hypersurfaces in step-2 Carnot groups are rectifiable in the classical sense, therefore the range of unrectifiability $0\le \alpha<1$ is sharp.
We also prove that the $(q-\frac12)$-dimensional Euclidean Hausdorff measure of every $C^1_\mathrm H$ hypersurface is zero. 
The previous results extend the seminal work of Kirchheim and Serra Cassano in the first Heisenberg group and settle a question concerning the singular surface they constructed, by proving that its $\frac52$-dimensional Euclidean Hausdorff measure is zero.

\end{abstract}

\maketitle

\setcounter{tocdepth}{1}
\tableofcontents

\newpage

\section{Introduction}

Intrinsic regular hypersurfaces appeared at the beginning of geometric measure theory in Carnot groups and play a crucial role in several developments. They represent the ``regular pieces'' that define the natural class of rectifiable sets in Carnot groups. These notions were introduced and applied by Franchi, Serapioni and Serra Cassano, who proved a De Giorgi's rectifiability theorem in Heisenberg groups, \cite{FSSC01}. Further developments in general Carnot groups can be found in \cite{FSSC5,AKLD2009TangSpGro,Marchi2014,Mag31,DonLDMV2022}.

It is not difficult to observe that a measurable subset of a Carnot group that has locally finite perimeter in the classical Euclidean sense is also a set of locally finite perimeter in the Carnot-group sense. This inclusion is also strict in general, as it can be seen in Heisenberg groups, using a countable family of Euclidean rectifiable sets, see \cite[Example~1]{FSSC01}.  The question of whether an intrinsic regular hypersurface may have an actual loss of Euclidean regularity was answered by Kirchheim and Serra Cassano
in the first Heisenberg group.
\begin{theorem}[Kirchheim and Serra Cassano, \cite{KirSer04}]\label{KSC reg}\label{T.KirSer04}
 In the Heisenberg group $\H^1$, there exists a $C^1_{\mathrm{H}}$ regular hypersurface whose Euclidean Hausdorff dimension is $5/2$.
\end{theorem} 
An extension of this theorem to $C^1_H$ regular hypersurfaces of $\H^n$ appeared in \cite{ArvilaMDeg2024}.
Heisenberg groups belong to the larger family of Carnot groups, that are simply connected, nilpotent Lie groups whose Lie algebra admits a stratification.
They were introduced by Folland under the name of ``stratified groups'', in the setting of linear subelliptic PDEs, \cite{Fol75}, see also \cite{BLU07Stratified}.
Subsequent developments of geometric measure theory in these groups are collected in \cite{SerraCassano2016}.
In Carnot groups, $C^1_\mathrm H$ regular hypersurfaces correspond to intrinsic regular hypersurfaces, or $\G$-regular hypersurfaces. In broad terms, these are the level sets of $C^1_\mathrm H$ functions with nonvanishing horizontal gradient. 
Notice that a $C^1_\mathrm H$ function need not be differentiable in Euclidean sense also on a set of positive measure, \cite[Section~4]{Mag6}, hence it can be highly irregular. 
We will be mainly interested in $C^{1,\alpha}_\mathrm H$ hypersurfaces, that are defined in analogous way,
where the defining function has an $\alpha$-H\"older continuous horizontal gradient (Definition~\ref{def: C^1,alpha function} and 
Definition~\ref{def: C^1,alpha hypersurface}).
 
We are interested in understanding Theorem~\ref{T.KirSer04} along three directions.
One concerns the role played by the group structure in the
construction of singularities. The other is the behaviour of Euclidean 
singularities as intrinsic regularity increases. 
The last one pertains to the sharpness of our estimates.
In sum, we wish to construct $C^{1,\alpha}_\mathrm H$ surfaces of a Carnot group, which have large Euclidean Hausdorff dimension (this should be read as the ``higher singularity'' allowed).

The first and second directions are contained in our first result.

\begin{theorem}\label{main theor codim k}
    Let $\G$ be a noncommutative Carnot group of topological dimension $q$, let $0\le \alpha<1$ and $1\leq k\leq m$. Then, there exists a $C^{1,\alpha}_{\mathrm{H}}$ co-horizontal surface $S$ of codimension $k$ that is not Euclidean $(q-k)$-rectifiable. More precisely, it satisfies 
    $\dim_\mathrm ES\geq q-k+\frac{1-\alpha}{2}$.
\end{theorem}

In the sequel, the symbol $q$ denotes the topological dimension of $\G$, $m$ denotes the dimension of the horizontal layer $V_1$, and $\dim_\mathrm E$ denotes the
Hausdorff dimension with respect to the Euclidean distance. $C^{1,\alpha}_H$ co-horizontal surfaces, also known as co-abelian intrinsic submanifolds, are regular intersections of $C^{1,\alpha}_H$ hypersurfaces (Definition~\ref{def:c1alphak}).
To construct Euclidean singularities on $C^{1,\alpha}_H$ hypersurfaces, the approach by Kirchheim and Serra Cassano can still be used in Carnot groups,
making Theorem~\ref{main theor codim k} new also in Heisenberg groups when $0<\alpha<1$. 

Let us outline some ideas about the proof of Theorem~\ref{main theor codim k}. In the $m$-dimensional cube $Q$ of unit side length, it is possible to find a subset $F_g\subset Q$ of large measure on which a $C^1$ function $g\colon Q\to\R$ has gradient equal to a vector field with nonvanishing
curl, \cite{Alberti91,Balogh03}. 
The vector field arises from the group operation, according to \eqref{def of a_h} and its nonvanishing curl
is precisely characterized by the Lie bracket generating
condition of $V_1$, as shown in Remark~\ref{rem: a_h not closed}.
The function $g$ can be constantly extended to a function $G$ on
$Q\times \R^{q-m-1}$.
As a result, the set of characteristic points of the graph of $G$ precisely contains the 
restriction of this graph to $F_g\times\R^{q-m-1}$. 
By the standard area formula, this set has positive $(q-1)$-dimensional Hausdorff measure, 
see Remark~\ref{rem:idea_sing_hyp}.
Then at each characteristic point we 
attach a level set 
of a highly irregular $\frac{1+\alpha}{2}$-H\"older function $\psi$,
along a vertical direction of degree $2$. The resulting set has 
an increased Euclidean Hausdorff dimension.
Combining the H\"older continuity of $\psi$, 
the special relation between the group operation and the associated left-invariant vector fields, the prescribed gradient of $g$,
and a suitable Whitney-type extension theorem,
this set can be extended to a Euclidean singular $C^{1,\alpha}_\mathrm H$ hypersurface.
Taking iterated intersections of the singular hypersurface with suitable hyperplanes, we obtain the higher codimensional $C^{1,\alpha}_\mathrm H$ surface with the Euclidean singularity as in the conclusion of Theorem~\ref{main theor codim k}.
Proving a version of this theorem in the strictly larger class 
of positively graded groups becomes a different and easier
problem, as we briefly discuss in Section~\ref{sec:graded}.

The existence of the H\"older continuous function $\psi$ above is an interesting problem in itself, which is also connected to a question raised 
by Kirchheim and Serra Cassano in \cite{KirSer04}, subsequently answered by Rajala in \cite{Rajala2026pr}. 
He constructs an $\alpha$-H\"older continuous function whose level sets all have positive $(1-\alpha)$-dimensional Hausdorff measure.
A slight modification of the construction for $\alpha=1/2$ yields an improved modulus of continuity with respect to $|x-y|^{1/2}$, 
still keeping all the level sets with Hausdorff dimension at least
$1/2$, see \cite[Lemma~3.5]{KirSer04}, or \cite[Proposition~2.6] {Rajala2026pr}. 
In our setting, we precisely need an improved modulus of continuity with respect to $|x-y|^\alpha$ for every $1/2<\alpha<1$, in order to apply the proper Whitney-type extension theorem. 
We follow Rajala's construction to obtain this modified statement,
see Section~\ref{section: app} for more details.

Surfaces of $C^{1,\alpha}_{\mathrm H}$ smoothness, with different definitions, have already been studied in other contexts, with various interesting results. It can be proved that submanifolds of such regularity are also Lipschitz rectifiable with respect to vertical (non-Euclidean) subgroups, \cite{DFO2022},
and can be characterized by suitable tangent $\alpha$-paraboloids, \cite{IduMai2024}.
In addition, boundedness of singular integral operators has been also studied on suitable classes of $C^{1,\alpha}_{\mathrm H}$ surfaces, \cite{CFO2019bound,CFO2025corrigendum,ChoLiZha2025}.
Theorem~\ref{main theor codim k} somehow confirms 
that the study of $C^{1,\alpha}_{\mathrm H}$ surfaces is genuinely non-Euclidean, in the sense that a purely Euclidean approach cannot be applied. 

Theorem~\ref{main theor codim k} for $k>1$ apparently provides the first examples of higher-codimension unrectifiable co-horizontal surfaces in arbitrary non-commutative Carnot groups. We notice that for $k=2$ in the first Heisenberg group $\H^1$, vertical curves are included. Other interesting constructions of irregular vertical curves are also available in Heisenberg groups. For instance, the $\lambda$-vertical curves recently introduced by Antonelli and Young can have different Hausdorff dimensions with respect to the distance of the group, according to \cite[Theorem~1.4]{AntYou2026}.

The third and last direction of our investigation studies whether
the lower estimates of Theorem~\ref{main theor codim k} can be 
further improved. As noted in \cite[Remark~3.3]{KirSer04}, the Euclidean Hausdorff dimension of the $C^1_{\mathrm H}$ surface in Theorem~\ref{KSC reg} is maximal, as a consequence of \cite[Theorem~1.1]{BalRicSer2003}.
An analogous observation shows the maximality
of the lower estimates in Theorem~\ref{main theor codim k} when $\G$ is a step-2 Carnot group, $\alpha=0$ and $S$ has codimension one. In other words, we obtain the equality $\dim_E S=q-\frac{1}{2}$.
This is a direct consequence of the Hausdorff dimension comparison estimates of \cite[Theorem~2.4]{BTW09}.  
These results do not rely on the $C^1_{\mathrm{H}}$ regularity of $S$.
For this reason, obtaining the upper estimates on the Euclidean Hausdorff dimension of $C^{1,\alpha}_{\mathrm H}$ hypersurfaces requires an additional method, that takes into account such intrinsic regularity. 

We are led to the second main result of the paper.

\begin{theorem}\label{main theorem upper estimate}
Let $\G$ be a Carnot group of step $2$ and let $0\leq\alpha\leq1$. Then every $C^{1,\alpha}_{\mathrm{H}}$ hypersurface $S$ satisfies $\dim_\mathrm E S\leq q-\frac{1+\alpha}{2}.$
\end{theorem}

The threshold $\alpha = 1$ actually has a stronger result: if $\G$ is of step $2$ and $S$ is a
$C^{1,1}_\mathrm H$ regular hypersurface, then $S$ is $(q-1)$-rectifiable in the classical sense, see Corollary~\ref{cor alpha=1}.
In other words, for $\alpha=1$ Euclidean singularities cannot be created and the equality $\dim_{\rm E}S=q-1$ always holds.
The equality $\dim_{\rm E}S=q-\frac{1+\alpha}{2}$ continues to hold for any highly irregular hypersurface $S$ constructed in Theorem~\ref{main theor codim k} and 
any $0\le \alpha<1$. Indeed, 
as a consequence of Theorem~\ref{main theorem upper estimate}, $S$  achieves the maximal Euclidean Hausdorff dimension allowed by its $C^{1,\alpha}_{\mathrm{H}}$ regularity, according to the following.
\begin{cor}
Let $\G$ be a noncommutative Carnot group of step $2$. For each $0\leq\alpha\leq1$ there exists a $C^{1,\alpha}_{\mathrm{H}}$ hypersurface $S$ such that $\dim_\mathrm E S= q-\frac{1+\alpha}{2}.$
\end{cor}
The geometric idea to prove Theorem~\ref{main theorem upper estimate} is to perform a local foliation of the ambient space in a small neighborhood of a fixed point in $S$.
We use the local intrinsic-graph structure of $S$ 
provided by the implicit function theorem, \cite{FSSC03}, 
along with an extra Euclidean regularity of the implicit function,
proved in Theorem~\ref{intrinsic parametrization is Holder}.
More precisely, the $C^{1,\alpha}_{\mathrm H}$ regularity of the defining function implies that the implicit function is $\frac{1+\alpha}{2}$-H\"older continuous with respect to the Euclidean distance.
Other related regularity results and their applications can be found in
\cite{CFO2019bound,DFO2022,CFO2025corrigendum,ChoLiZha2025}.

The foliation is obtained by right-translating small pieces of $S$ along a horizontal vector where the directional derivative of the defining function does not vanish. We show that this one-parameter foliation map is 
locally invertible and regular, where the right translated
leaves $S_t$ are all locally bi-Lipschitz equivalent.
Then the standard dimensional estimates for a.e. level sets of H\"older continuous functions are sufficient to obtain the upper bound $q-\frac{1+\alpha}{2}$ on the Hausdorff dimension of the leaf $S_0=S$. 
The geometric picture and the argument of this approach are inspired 
by the proof of \cite[Theorem 2.16]{Mag5}, under different regularity assumptions. 

The previous results also leave several interesting open questions for future research. One of them is to find the maximal Euclidean Hausdorff dimension for $C^{1,\alpha}_{\mathrm H}$ hypersurfaces in higher step Carnot groups, possibly also for higher codimensional co-horizontal $C^{1,\alpha}_{\mathrm H}$ surfaces.
Another problem is whether one can construct a $C^{1,\alpha}_{\mathrm{H}}$ hypersurface $S$ in a step-$2$ Carnot group with
$\HH^{q-\frac{1+\alpha}{2}}(S)>0$.
This question has an answer for $\alpha=0$, that can be
obtained combining the method of Theorem~\ref{main theorem upper estimate} and the $\frac12$-little H\"older continuity of the graphing function in step-2 Carnot groups, see \cite{DiD21} or \cite{ADDDLD24}. 

\begin{cor}
If $\G$ is a Carnot group of step $2$ and $S\subset\G$ is a $C^1_{\mathrm{H}}$ hypersurface, then $\HH^{q-\frac{1}{2}}(S)=0.$
\end{cor}
The previous corollary is restated and proved in Corollary~\ref{cor:Hq-1/2=0}. When $\G$ is the Heisenberg group $\H^1$, this result completes the picture in \cite{KirSer04} by answering a question left open there: the $\frac52$-dimensional Euclidean Hausdorff measure of the surface constructed in Theorem~\ref{T.KirSer04} must vanish.

\section{Preliminaries}\label{sect:preliminaries}

Let $\G$ be a real, finite dimensional, nilpotent, and simply connected Lie group. We say that $\G$ is a Carnot group, or a
stratified group, of step $s$ if its Lie algebra 
${\rm Lie}(\G)$ satisfies the following conditions. 
The stratification of ${\rm Lie}(\G)$ is given by
\begin{equation*}
{\rm Lie}(\G)=V_1\oplus V_2\oplus\cdots\oplus V_s,\quad
\text{where $\;[V_1,V_i]=V_{i+1}$ for every $i\ge1$},
\end{equation*}
$V_s\neq\{0\}$ for some positive integer $s$, and $V_i=\{0\}$ whenever $i>s$.
For every $i=1,\ldots,s$,
we define 
\begin{equation*}
n_i=\dim V_i\ge1,\quad\text{and} \quad m_i=\sum_{j=1}^i n_j,
\end{equation*}
along with $m_0=0$. In the sequel we denote $m_1$ simply by $m$.
Moreover, we denote by
\begin{equation*}
    q=\sum_{j=1}^s n_j\quad\text{and}\quad Q=\sum_{j=1}^s j n_j
\end{equation*}
the topological and Hausdorff dimension of $\G$, respectively. The latter dimension refers to the natural distance
on $\G$, called {\em homogeneous distance}, that will be defined below.

Under our hypotheses on $\G$, the exponential mapping $\exp:\Lie(\G)\to\G$ is bianalytic, therefore we can assume that
$\G$ has an additional linear structure, induced by $\Lie(\G)$.
Then we fix a basis $e_1,\ldots,e_q$ of $\G$ adapted to the stratification and consider the corresponding adapted
basis $(X_1,\ldots,X_q)$ of $\Lie(\G)$ such that 
$X_i(0)=e_i$ for every $i=1,\ldots,q$. 
The basis $(X_1,\ldots,X_m)$ of $V_1$ is the frame of left-invariant, {\em horizontal vector fields}.
Any linear combination of $X_1(x),\ldots,X_m(x)$ defines a 
{\em horizontal vector} of $T_x\G$. 

Exploiting the linear structure of $\G$ we introduce 
\emph{graded coordinates of the first kind}
$(x_1,\ldots,x_q)$, that is
$$
x=\exp(x_1X_1+\ldots+x_qX_q)\in\G.
$$
Throughout the paper, we work in this special coordinate system and identify the point $x$ with its coordinates
$(x_1,\ldots,x_q)\in\R^q$. We will also use the notation $x=\sum_{i=1}^qx_ie_i$.

For $1\leq  i\leq q$ we denote by $d_i\in\N$ the unique integer such that 
$m_{d_i-1}<i\le m_{d_i}$, that is the {\em degree} of the coordinate $x_i$. 
Then for any $\lambda>0$ we can introduce the 
\textit{intrinsic dilations} $\delta_\lambda \colon \G \to \G$,
which are defined by
\[
\delta_\lambda(x_1, \dots, x_q) = (\lambda^{d_1} x_1, \dots, \lambda^{d_q} x_q).
\]
The group law is determined by the Baker-Campbell-Hausdorff formula, in short BCH formula. For every $x,y\in\G$, it can be represented as
\begin{equation}\label{eq:BCHbasisfree}
    x  y =\sum_{i=1}^s c_i(x,y)=x+y+\sum_{j=m+1}^q\mathcal{Q}_j(x,y)e_j\in\G,
\end{equation}
 where $c_1(x,y)=x+y$ and each $c_i$ is a linear combination of $(i-1)$-iterated Lie brackets of $x$ and $y$, while each $\QQ_j$ is a homogeneous polynomial of degree $d_j$ with respect to the dilations of $\G$. Moreover, the terms $c_i$ satisfy the following recursive formula
\begin{equation}\label{eq:recursiveBCH}
\begin{split}
&(i + 1) c_{i+1}(X, Y) = \frac{1}{2} [X - Y, c_i(X, Y)] \\
&+\sum_{\substack{l \ge 1 \\ 2l \le i}} K_{2l} \sum_{\substack{k_1, \dots, k_{2l} > 0 \\ k_1 + \dots + k_{2l} = i}} [c_{k_1}(X, Y), [\dots, [c_{k_{2l}}(X, Y), X + Y],] \dots ],    
\end{split}
\end{equation}
for each $i=1,\ldots,s-1$, see \cite[Lemma 2.15.3]{Var84Lie}.

The left-invariant horizontal vector fields $X_1,\ldots,X_m$ 
span the first layer $V_1$ of ${\rm Lie}(\G)$ and Lie-generate
$\Lie(\G)$. Having a system of graded coordinates fixed, combined
with the BCH formula, yields a precise analytic expression
of the left-invariant vector fields
\begin{equation*}
    X_i(x)=\frac{d}{dt}(x te_i)\Big|_{t=0}=\partial_{x_i}+\sum_{j=m+1}^q a_{ij}(x)\partial_{x_j}.
\end{equation*}
Here $(\partial_{x_1},\ldots,\partial_{x_q})$ is the standard basis of vector fields of $\R^q$ and their evaluation at the point $x$ is understood. Moreover, we have
\begin{equation}\label{def of a_ij}
   a_{ij}(x)\coloneqq \frac{d}{dt}\QQ_j(x,te_i)\Big|_{t=0}.
\end{equation}

The so-called Lie bracket generating condition of $V_1$ 
is the characterizing property of a stratified group $\G$,
that gives special properties to the polynomials $a_{ij}$.
To emphasize the important role of this assumption in the
results of this paper, 
we also consider the strictly larger family of  
{\em positively graded Lie groups}. 
The only difference in a positively graded group $\G$ 
with respect to a stratified group is that the condition
$[V_1,V_i]=V_{i+1}$ is replaced by the weaker condition
\[
[V_i,V_j]\subset V_{i+j}
\]
for every $i,j\ge1$, with $\Lie(\G)=V_1\oplus\cdots V_s$ 
and $V_i\neq\{0\}$ for every $i=1,\ldots,s$.
In a sense, for positively graded groups the group operation
``may degenerate'', according to the following proposition.

\begin{prop}\label{prop: not stratified iff Rj=0}
A positively graded group $\G$ is not stratified if and only if there exist a graded basis $(w_1,\ldots,w_q)$ of $\G$ and $h>m$ such that representing the group operation as
\begin{equation*}
    xy=x+y+\sum_{j=m+1}^q R_j(x,y)\, w_j,
\end{equation*}
the polynomial $R_h$ is identically zero.
\end{prop}
\begin{proof}
For the if part, we may assume that there exists a graded basis $(w_1,\ldots,w_q)$ such that $R_l\equiv0$ for some $l\geq m+1$. 
We recall the form of the left-invariant vector field
\begin{equation}\label{eq: def of Xi wrt base wi}
X_i(x)=\partial_{w_i}+\sum_{j=m_{d_i}+1}^q b_{ij}(x) \partial_{w_j}
\end{equation}
associated with the basis $(w_1,\ldots,w_q)$, for every $i=1,\ldots,q$, where
\begin{equation}\label{eq: def of bij}
   b_{ij}(x)\coloneqq \frac{d}{dt}R_j(x,tw_i)\Big|_{t=0}.
\end{equation}
By our assumptions, we have $b_{il}\equiv0$ whenever $1\le d_i<d_l$.
As a consequence, if $(w^*_1,\ldots,w^*_q)$ denotes the dual basis of $(\partial_{w_1}|_0,\ldots,\partial_{w_q}|_0)$,
we obtain $w_l^*(X_i(0))=0$ and $w_l^*([X_i,X_j](0))=0$ for every $i,j=1,\ldots,m$. Therefore, by induction, we have $w_l^*(Y(0))=0$ for every $Y\in\Lie$-$\spann(V_1)$. It follows that $X_l\not\in\Lie$-$\spann(V_1)$, thus $\G$ is not stratified.

For the converse, assume that $\G$ is not stratified. Let $i^*\geq2$ be the smallest integer such that
$$V_{i^*}\setminus [V_1,V_{i^*-1}]=V_{i^*}\setminus \Lie\text{-}\spann(V_1)\neq\emptyset.$$
Let $m_{i}=\sum_{j=1}^{i}\dim V_j$ for all $i=1,\ldots,s$, and set  
    $$k=\dim (V_{i^*}\cap\Lie\text{-}\spann(V_1))<\dim V_{i^*}.$$
Let $W_i=\exp V_i\subset\G$ for every $i=1,\ldots s$, and observe that each $W_i$ inherits a vector space structure from $V_i$.
We select linearly independent vectors $w_{m_{i^*-1}+1},\ldots,w_{m_{i^*-1}+k}\in W_{i^*}$ such that 
    \begin{equation*}\label{eq: liespan1 in Vi}
       \spann \{w_{m_{i^*-1}+1},\ldots,w_{m_{i^*-1}+k}\}=W_{i^*}\cap\Lie\text{-}\spann(W_1).
    \end{equation*}
We complete them to a graded basis $(w_1,\ldots,w_q)$ of $\G=W_1\oplus \cdots \oplus W_{s}$, therefore 
 \begin{equation}\label{eq: Vi othogonal to liespanV1}
    \spann\{w_{m_{i^*-1}+k+1},\ldots,w_{m_{i^*}}\}\cap\Lie\text{-}\spann(W_1)=\{0\}.
 \end{equation}
Let $(w^*_1,\ldots,w^*_q)$ be the dual basis of $(w_1,\ldots,w_q)$. 
Then, for each $j=m_{i^*-1}+1,\ldots,m_{i^*}$ the BCH formula yields
\begin{equation}\label{eq: R_j}
\begin{split}
    R_{j}(x,y)&=w^*_j\left(\sum_{i=2}^s c_i(x,y)\right)=\sum_{i=2}^{i^*} w^*_j\circ c_i(x,y)\\
    &=\sum_{i=2}^{i^*} w^*_j\circ c_i\big(\pi_{i^*-1}(x),\pi_{i^*-1}(y)\big),
\end{split}
\end{equation}
where $\pi_{i^*-1}\colon\G\to W_1\oplus\cdots \oplus W_{i^*-1}$ is the canonical projection map. 
By the minimality of $i^*$ we have $W_1\oplus\cdots \oplus W_{i^*-1}\subset\Lie\text{-}\spann(W_1)$, therefore $\pi_{i^*-1}(x),\pi_{i^*-1}(y)\in \Lie\text{-}\spann(W_1)$. 
By \eqref{eq:recursiveBCH}, each $c_i(\pi_{i^*-1}(x),\pi_{i^*-1}(y))$ is a linear combination of iterated commutators of $\pi_{i^*-1}(x)$ and $\pi_{i^*-1}(y)$ for all $i=2,\ldots,i_*$, therefore
$$
c_i(\pi_{i^*-1}(x),\pi_{i^*-1}(y))\in\Lie\text{-}\spann(W_1).
$$
Thus, by \eqref{eq: Vi othogonal to liespanV1} and \eqref{eq: R_j}, we conclude that $R_j\equiv 0$ for each $j=m_{i^*-1}+k+1,\ldots,m_{i^*}.$
\end{proof}

\begin{remark}
    In particular, the previous proposition implies that each polynomial $\QQ_j$ of \eqref{eq:BCHbasisfree} is nonvanishing for every $j\geq m+1$.
\end{remark}
In any positively graded group $\G$, it is possible to define the so-called 
{\em homogeneous norm} $\|\cdot\|:\G\to[0,+\infty)$, that satisfies the
properties  
$\|\delta_rx\|=r\|x\|$ for all $r>0$, $x\in\G$, and 
$\|xy\|\le \|x\|+\|y\|$ for all $x,y\in\G$, see \cite{HebSik90}.
Setting $d(x,y)=\|x^{-1}y\|$ yields the so-called homogeneous distance. Indeed, we define a {\em homogeneous distance} 
to be a distance $d:\G\times\G\to[0,+\infty)$ satisfying
the conditions $d(zx,zy)=d(x,y)$ and $d(\delta_rx,\delta_ry)=rd(x,y)$
whenever $x,y,z\in\G$ and $r>0$.
It is not difficult to observe that all homogeneous distances are bi-Lipschitz equivalent to each other. 
Thus, for our purposes, it suffices to consider a fixed homogeneous norm defined as follows
\begin{equation*}
    \|x\|_\G=\max_{i=1,\ldots,q}c_i|x_i|^{1/d_i},
\end{equation*}
where $0<c_i\le 1$ are suitable constants and $c_1=1$, see for instance \cite[Section~2.2]{SerraCassano2016}.
This norm clearly defines the homogeneous distance
\begin{equation}\label{def of d}
    d(x,y)=\|x^{-1}y\|_\G.
\end{equation}
It is important to compare the homogeneous distance with the Euclidean distance. For any bounded open set $\Omega\subset \G$ there exist $C_1,C_2>0$ such that for every $x,y\in\Omega$
\begin{equation}\label{comparison between d_E and d_CC}
        C_1|x-y|\leq d(x,y)\leq C_2|x-y|^{\frac{1}{s}}.
\end{equation}

\begin{defn}\label{def: C^1,alpha function}
Let $\Omega\subset\G$ be an open set and let $(X_1,\ldots,X_m)$ be a basis of $V_1$. We say that a continuous function $f\colon\Omega\to\R$ is $C^1_{\mathrm{H}}$ if the {\em horizontal gradient} 
$$
\nabla_{\mathrm{H}} f=(X_1f,\ldots,X_m f)
$$
pointwise exists everywhere and it is continuous in $\Omega\subset\G$. The linear space of all 
$C^1_\mathrm H$ functions on $\Omega$ is denoted by $C^1_\mathrm H(\Omega)$. 
If in addition there exists $C>0$ such that
\[ 
|\nabla_\mathrm H f(x)-\nabla_\mathrm H f(y)|\le C\, d(x,y)^\alpha 
\]
for every $x,y\in\Omega$ and for some $0<\alpha\leq 1$, 
we say that $f:\Omega\to\R$ is $C^{1,\alpha}_\mathrm H$ on $\Omega$. We denote by $C^{1,\alpha}_\mathrm H(\Omega)$ the linear space of all $C^{1,\alpha}_\mathrm H$ functions
in $\Omega$ and set $C^1_\mathrm H(\Omega)=C^{1,0}_\mathrm H(\Omega)$.
\end{defn}

The class of $C^{1,\alpha}_\mathrm H$ functions defines the corresponding class of $C^{1,\alpha}_\mathrm H$ hypersurfaces.

\begin{defn}\label{def: C^1,alpha hypersurface}
We say that $S\subset \G$ is a {\em $C^{1,\alpha}_{\mathrm{H}}$ hypersurface}, $0\leq\alpha\leq 1$, if for every $x\in S$ there exist an open neighborhood $\Omega\subset\G$ of $x$ and a $C^{1,\alpha}_{\mathrm{H}}$ function $f:\Omega\rightarrow \R$ such that $S\cap \Omega=\{y\in \Omega \colon f(y)=0\}$ and $\nabla_{\mathrm{H}}f(y)\neq 0$ for any $y\in \Omega$.
\end{defn}
Vector-valued $C^{1,\alpha}_\mathrm H$ functions are defined in an analogous way.
\begin{defn}
Let $\Omega\subset\G$ be open, $j\in\N^+$, and $0\le\alpha\leq1$. 
A function $f\colon\Omega\to\R^j$ is $C^{1,\alpha}_{\mathrm{H}}$ 
if so is each of its components $f_i$, $1\le i\le j$.
The {\em h-Jacobian matrix} of $f$ at $x\in\Omega$ is defined as
the following $j\times m$ matrix
\[
\nabla_\mathrm H f(x)=\left(\begin{array}{c}
     \nabla_\mathrm H f_1(x)  \\
     \nabla_\mathrm H f_2(x)  \\
     \vdots  \\
     \nabla_\mathrm H f_j(x)
\end{array}\right).
\]
The linear space of all $C^{1,\alpha}_\mathrm H$ functions on $\Omega\subset\G$ and taking values in $\R^j$ is denoted by $C^{1,\alpha}_\mathrm H(\Omega,\R^j)$.
\end{defn}

Next, we recall some basic tools that will be used in the sequel.
The first is a special version of a general Whitney's extension theorem, proved in \cite{PupVod2006}.

\begin{theorem}[Whitney's $C^{1,\alpha}_\mathrm H$ extension theorem]\label{whitney alpha}
    Let $F \subset \mathbb{G}$ be a closed set, let $\alpha\in(0,1]$ and $f\colon F\to \R$, $\varphi\colon F\to \R^m$ satisfying the following properties: \\
    there exists a constant $M$ such that
    \begin{enumerate}[label=\emph{(\roman*)}]
        \item[(i)] $|f(x)|, |\varphi(x)| \leq M \quad \text{for every } x\in F;$
    \item[(ii)] $|f(x) - f(y) - \langle \varphi(x), \pi(y^{-1}x) \rangle| \leq M d(x,y)^{1+\alpha} \quad \text{for every } x,y \in F;$
    \item[(iii)] $|\varphi(x) - \varphi(y)| \leq M d(x,y)^{\alpha} \quad \text{for every } x,y \in F;$
    \end{enumerate}
    where we have used the projection $\pi(x)=\pi(x_1,\ldots,x_q)=(x_1,\ldots,x_m)$ with respect to graded coordinates
    and the Euclidean scalar
    product $\langle\cdot,\cdot\rangle$. Then there exists an extension $\bar{f} : \G\to \mathbb{R}$ of $f$, such that $\bar{f} \in C_{\mathrm{H}}^{1,\alpha}(\G)$ and
\begin{equation*}
    \varphi(x) = \nabla_{\mathrm{H}} \bar{f}(x) \quad \text{for all } x \in F.
\end{equation*}
\end{theorem} 

The notion of \textit{intrinsic graph} naturally appears in connection with the implicit function theorem in Carnot groups. 
As it occurs in the Euclidean case, the theorem requires a 
factorization to represent the level set as a suitable graph. For a Carnot group $\G$, we consider the factorization with
respect to two homogeneous subgroups $\W\subset\G$ and 
$\V\subset\G$. Precisely, we say that $\W$ and $\V$ are
\textit{complementary subgroups} if 
$\W\cap\V=\{0\}$ and $\G=\W\V$. 
We may characterize the complementary subgroups $\W$ and $\V$ by the group projections
\[
\pi_\W:\G\to\W \quad \text{and}\quad \pi_\V:\G\to\V
\]
such that for every $x\in\G$ we have the unique representation $x=\pi_\W(x)\pi_\V(x)$.

For any subset $U\subset\W$ we define the \textit{intrinsic graph} of $\phi\colon U\subset\W\to\V$ as the set 
\begin{equation}\label{def intrinsic graph}
    \graph(\phi)\coloneqq \{w\phi(w)\colon w\in U\}.
\end{equation}
The implicit function theorem essentially shows that 
regular level sets have the intrinsic graph form of \eqref{def intrinsic graph}; 
see for instance \cite[Theorem 2.1]{FSSC03}. We state it in the following version.

\begin{theorem}[Implicit function theorem]\label{implicit function theorem}
Let $\Omega\subset\G$ be open, let $f\in C^1_{\mathrm{H}}(\Omega)$ and assume that 
$\nabla_Hf(x)\neq0$ for some $x\in f^{-1}(f(x))=S$.
We fix a graded basis $(e_1,\ldots,e_q)$ of $\G$ 
with associated basis $(X_1,\ldots,X_q)$ of left-invariant vector fields in $\Lie(\G)$.
We have $X_jf(x)\neq0$ for some $j\in\{1,\ldots,m\}$.
Let us define the homogeneous subgroups 
\[
\V=\spann\{e_j\}\quad \text{and} \quad \W=\spann\{e_1,\ldots,e_{j-1},e_{j+1},\ldots,e_q\}.
\]
Then there exist open neighborhoods $\Omega'\subset\G$ of $x$ and $U\subset\W$ of $\pi_{\W}(x)$, and a continuous function $\phi\colon U\to\V$ such that
    \begin{equation*}
        S\cap \Omega'=\graph{\phi}=
\{w\phi(w) \colon w\in  U\}.
    \end{equation*}
\end{theorem}
To estimate the Hausdorff dimension of our surfaces, we need the well-known coarea inequality, or Eilenberg inequality. A first general version of this
inequality was proved by Federer in \cite[2.10.25]{Federer69}, under specific assumptions on the target, or on the measure of level sets.
These assumptions were removed by the work of 
Davies, \cite{Davies1970Increasing}.
We address the reader to the work of Esmayli and Haj\l asz  \cite{EsmHaj2021} for a different proof and an interesting historical account on the result.

\begin{theorem}[Eilenberg inequality]\label{lemma:Eilenberg}
Let $X$ and $Y$ be two arbitrary metric spaces, let $E \subset X$ be any subset and fix $0 \leq t \leq s < \infty$.
Then, for any Lipschitz map $f : X \to Y$, it holds \begin{equation}\label{ineq:Eilenberg}
\int_Y^*\mathcal{H}^{s-t} \big( f^{-1}(y) \cap E \big) \, d\mathcal{H}^t(y) 
\leq (\operatorname{Lip} f)^t \, \frac{\omega_{s-t} \, \omega_t}{\omega_s} \, \mathcal{H}^s(E).
\end{equation}
Moreover, if $X$ is boundedly compact, namely, bounded and closed sets in $X$ are compact, $E$ is $\mathcal{H}^s$-measurable, and with $\mathcal{H}^s(E) < \infty$, then
\begin{equation}\label{eq:integrand}
 y \mapsto \mathcal{H}^{s-t} \big( f^{-1}(y) \cap E \big)
\end{equation}
is $\mathcal{H}^t$-measurable and the upper integral 
of \eqref{ineq:Eilenberg} coincides with the integral of \eqref{eq:integrand} with respect to $\mathcal H^t$.
\end{theorem}
In the previous theorem, $\mathcal H^\beta$ denotes 
the $\beta$-dimensional Hausdorff measure in metric spaces, 
with $\beta>0$.
Since $\G$ has a linear structure, we also assume that an
auxiliary scalar product is fixed. 
If not otherwise stated, in the sequel the symbol 
$\mathcal{H}^\beta$ will be used to indicate
the $\beta$-dimensional Hausdorff measure with respect to the distance given by the fixed scalar product.
The notation $\dim_\mathrm E$ denotes the Hausdorff dimension
with respect to the same distance.
We will also refer to $\mathcal H^\beta$ as the $\beta$-dimensional Euclidean Hausdorff measure.

\section{Unrectifiable surfaces with intrinsic regularity}

In this section, we will construct $C^{1,\alpha}_\mathrm H$ surfaces
in a Carnot group such that their Euclidean Hausdorff dimension is strictly greater than their topological dimension. 
We start with the following definition.

\begin{defn}\label{def:c1alphak}
    Let $\G$ be a Carnot group and $1\leq k\leq m$. We say that $S\subset \G$ is a \textit{co-horizontal (or co-abelian) $C^{1,\alpha}_{\mathrm{H}}$ surface of codimension $k$} if, for every $x\in S$, there exists a neighborhood $\Omega$ of $x$ and a map $f\in C^{1,\alpha}_{\mathrm{H}}(\Omega,\R^k)$ such that 
    \begin{equation*}
        S\cap\Omega=\{y\in \Omega\colon f(y)=0\},
    \end{equation*}
    and the h-Jacobian matrix $\nabla_\mathrm H f(y)$ has rank $k$ for every $y\in \Omega$.
\end{defn}
The terminology arises from \cite{ADDDLD24} and \cite{KozPhD}.
It can be checked that the h-Jacobian is also the matrix
of the Pansu differential, so we can apply 
\cite[Theorem 3.5.1]{KozPhD} and conclude that 
the Hausdorff dimension of a co-abelian surface has
Hausdorff dimension $Q-k$ with respect to a 
homogeneous distance $d$. We have denoted again by 
$Q$ the Hausdorff dimension of $\G$ with respect to $d$.

Throughout the section, we identify $\G$ with $\R^q$, as in Section~\ref{sect:preliminaries}. Accordingly, any two points $x,y\in\G$ are written in coordinates as 
$x=(x_1,\ldots,x_q)$ and $y~=~(y_1,\ldots,y_q)$, and we denote their horizontal components by 
\begin{equation*}
    x^1=(x_1,\ldots,x_m),\quad y^1=(y_1,\ldots,y_m).
\end{equation*} 
In the sequel, we fix an index $h\in\{m_1+1,\ldots,m_2\}$. 
With a slight abuse of notation, 
the polynomial $\QQ_h(\sum_{i=1}^mx_ie_i,\sum_{i=1}^m y_ie_i)$ of the BCH formula will be simply denoted $\QQ_h(x,y)$. 
The special choice of $h$ ensures that 
$\QQ_h(x,y)$ is antisymmetric and bilinear,
and depends only on the horizontal components $x^1$ and $y^1$.
More precisely,
\begin{equation*}
    \QQ_h(x,y)=\QQ_h(x^1,y^1)=\frac{1}{2}[x^1,y^1]_h. 
\end{equation*}
By a Taylor expansion around $y=0$, we obtain
    \begin{equation*}\label{Q_h=a_h*y}
        \begin{split}        \QQ_h(x^1,y^1)&=\QQ_h(x^1,0)+\langle\,\nabla_{y} \QQ_h(x^1,0),y^1\rangle \\        &=\sum_{i=1}^m y_i\frac{\partial \QQ_h}{\partial y_i}(x^1,0)=\sum_{i=1}^m a_{ih}(x^1)\,y_i,
        \end{split}
    \end{equation*}
    where $a_{ih}$ are defined as in \eqref{def of a_ij}. Moreover, notice that
    \begin{equation*}
        0=\QQ_h(x^1,x^1)=\sum_{i=1}^m a_{ih}(x^1)\, x_i.
    \end{equation*}
    Thus, if we define 
    \begin{equation}\label{def of a_h}
        a_h(x^1)\coloneqq (a_{1h}(x^1),\ldots,a_{mh}(x^1)),
    \end{equation}
    combining the three previous identities, we obtain
    \begin{equation}\label{eq: Q_h=<a_h,x-y>}
        \QQ_h(x,-y)=\langle a_h(x^1),\,x^1-y^1 \rangle.
    \end{equation}

\begin{remark}\label{rem: a_h not closed}
It is interesting to notice that the vector field $a_h$ 
in \eqref{def of a_h} cannot be irrotational if 
$\G$ is a noncommutative Carnot group. 
In equivalent terms, we wish to observe that 
the $1$-form associated with $a_h$ is not closed. 
Let us consider
\begin{equation}\label{ext derivative}
    d\left[a_{1h}(x^1)dx_1+\ldots+a_{mh}(x^1)dx_m\right]=\sum_{1\leq i<j\leq m} \left(\frac{\partial a_{jh}}{\partial x_i}(x^1)-\frac{\partial a_{ih}}{\partial x_j}(x^1)\right)dx_i\wedge dx_j.
\end{equation}
Taking into account \eqref{def of a_ij}, we have 
\begin{equation*}
    \frac{\partial a_{jh}}{\partial x_i}(x^1)=\frac{\partial}{\partial{x_i}}\frac{\partial}{\partial{y_j}}\mathcal{Q}_h(x^1,y^1)=-\frac{\partial}{\partial{x_j}}\frac{\partial}{\partial{y_i}}\mathcal{Q}_h(x^1,y^1)=-\frac{\partial a_{ih}}{\partial x_j}(x^1).
\end{equation*}
If by contradiction we assume that all the coefficients on the 
right-hand side of \eqref{ext derivative} vanish, 
then the previous identity gives 
\begin{equation*}
    \frac{\partial a_{jh}}{\partial x_i}(x^1) = 0 \quad\text{for all } i,j=1\ldots m.
\end{equation*}
It follows that every component $a_{jh}$ must vanish identically, hence $\mathcal{Q}_h$ is identically zero. By our choice $h> m$, in view of Proposition~\ref{prop: not stratified iff Rj=0}, this cannot occur in a noncommutative Carnot group.
\end{remark}

The previous remark motivates the application of
a result by Balogh \cite{Balogh03}, which strengthens the previous version by Alberti \cite{Alberti91}. 
It allows to find a $C^1$ potential $g$ whose gradient 
coincides with $a_h$ on a set of positive measure. 
Due to Remark \ref{rem: a_h not closed}, the condition $\nabla g=a_h$ clearly cannot hold on an open set.
This result is an important tool to prove
Theorem~\ref{main theor codim k}.
We state it in the next lemma, where the specific modulus of
continuity of the gradient of $g$ can be found at the 
end of the proof of \cite[Theorem~4.1]{Balogh03}.

\begin{lemma}\label{lemma:Balogh}
Let $Q=[0,1]^m \subset \mathbb{R}^m$ be the closed unit cube and  $X\colon Q\to\R^m$ be a $C^1$ smooth vector field. Then for any $\eps>0$, there exist a $C^1$ function $g : Q \to \R$, a closed subset $F_g\subset Q$ with $\mathcal{L}^m(F_g) > 1-\eps$, and a constant $0<K < \infty$ such that:
\begin{enumerate}[label=(\roman*)]
\item $\nabla g(z) = X(z)$ for all $z \in F_g$,
\item $|\nabla g(z_1) - \nabla g(z_2)| \le K \big(1 + |\log |z_1-z_2||\big)^K |z_1-z_2|$ for all $z_1, z_2 \in Q$.
\end{enumerate}
\end{lemma}

Another crucial tool is the following.
\begin{lemma}\label{lemma:Rajala-alpha}
For every $\gamma,\beta\in(0,1)$, there exists a function $\psi: \mathbb{R} \to \mathbb{R}$ such that
\begin{enumerate}[label=\emph{(\roman*)}]
\item For all $t \in \mathbb{R}$, $\dim_\mathrm E\psi^{-1}(t)\geq1-\gamma$;
\item For each $M \ge 1$, we have
\begin{equation*}
\lim_{r \to 0^+} \frac{e^{M\left|\log r\right|^\beta}}{r^{\gamma}} \sup \big\{ |\psi(x)-\psi(y)| : |x-y|\le r \big\} = 0.
\end{equation*}
\end{enumerate}
\end{lemma}
The previous lemma is a straightforward consequence of
Theorem~\ref{thm: rajala holder+}.
We start from $\tilde \psi\colon[0,1]\to[0,1]$ given
by this theorem. 
We define $\psi(x)=k+\tilde \psi(x-k)$ for every $k\in\Z$ and $x\in[k,k+1]$. Since $\tilde{\psi}(0)=0$ and $\tilde{\psi}(1)=1$, the function $\psi$ is continuous. Moreover, property (i) is preserved under translations.
    For property (ii), we can restrict to $r\leq 1$. Let $x,y\in\R$ with $|x-y|\leq r$. If $x,y$ lie in the same interval $[k,k+1]$, property (ii) follows directly from Theorem \ref{thm: rajala holder+}. 
    If instead $x\in [k-1,k]$ and $y\in[k,k+1]$, then the triangle inequality gives 
    $$|\psi(x)-\psi(y)|\leq |\psi(x)-\psi(k)|+|\psi(k)-\psi(y)|\leq 2\sup_{\substack{|z-w|\leq r\\z,w\in [0,1]}}|\tilde \psi(z)-\tilde \psi(w)|,$$ 
    hence concluding the proof of (ii) in Lemma~\ref{lemma:Rajala-alpha}.

The following lemma represents an important step in
the proof of Theorem~\ref{main theor codim k}.

\begin{lemma}\label{lem: extension}
Let $h\in\{m+1,\ldots,m_2\}$, $0<\eps<1$, $Q=[0,1]^m$, and $T\subset\R^{q-m}$ be an arbitrary nonempty compact.
We consider the vector field
\begin{equation}\label{eq:Xah}
X= a_h|_Q:Q\to\R^m,
\end{equation}
where $a_h$ is defined in \eqref{def of a_h}. 
Let $g:Q\to\R$ and $F_g\subset Q$ be as in Lemma~\ref{lemma:Balogh},
with $X$ as in \eqref{eq:Xah}.
Let $0\le \alpha<1$ and consider any $\psi\colon\R\to\R$
that satisfies conditions (i) and (ii) of
Lemma~\ref{lemma:Rajala-alpha},
with parameters $\gamma=\frac{1+\alpha}{2}$ and $0<\beta<1$.
Then there exists $\tilde f\in C^{1,\alpha}_{\mathrm{H}}(\G)$ such that for every $x\in F_g\times T$ we have
\begin{enumerate}[label=(\roman*), font=\upshape]
    \item $\tilde f(x)=\psi(x_h-g(x^1))$;
    \item $\nabla_{\mathrm{H}} \tilde f(x)=0$.
\end{enumerate}
\end{lemma}
\begin{proof}
We define $F=F_g\times T$ and $f^*\colon F\to\R$ as 
\begin{equation}\label{eq:f^*x}
    f^*(x)\coloneqq \psi(x_h-g(x^1)).
\end{equation}
 We aim to apply Theorem~\ref{whitney alpha} to $f^*$ and $\varphi\equiv0$ on $F$.
By the compactness of $F$ and the special choice of $\varphi$, 
conditions (i) and (iii) of Theorem~\ref{whitney alpha} are immediately satisfied.  It suffices to verify condition (ii), that is,
\begin{equation}\label{whitney condition}
    |f^*(x) - f^*(y)| \leq M d(x,y)^{1+\alpha} \quad \text{for every } x,y\in F.
\end{equation}
By (ii) of Lemma~\ref{lemma:Rajala-alpha}, there exists $C>0$ such that
\begin{equation}\label{eq: f(p)-f(p')}
    |f^*(x)-f^*(y)|\leq C\frac{|x_h-g(x^1)-(y_h-g(y^1))|^{\frac{1+\alpha}{2}}}{e^{M|\log |x_h-g(x^1)-(y_h-g(y^1))||^\beta}},
\end{equation}
taking into account \eqref{eq:f^*x}. 
By the antisymmetry of $\QQ_h$ we have
\begin{gather*}
    \left|x_h-g( x^1)-(y_h-g(y^1))\right|\\
    \leq \left|x_h-y_h+\QQ_{h}(-y,x)\right|+
    \left|-g( x^1)+g(y^1)+\QQ_{h}(x,-y)\right|.
\end{gather*}
By definition of $d$ in \eqref{def of d}, since the $h$-th component has degree $2$, 
we obtain
\begin{equation*}\label{eq:A}
    \left|x_h-y_h+\QQ_{h}(-y,x)\right|\leq \left(\frac{d(x,y)}{c_h}\right)^2.
\end{equation*}
Now we use \eqref{eq: Q_h=<a_h,x-y>} and Lemma~\ref{lemma:Balogh}
to get 
\begin{equation*}\label{eq:B}
\begin{split}
    &|g(x^1)-g(y^1)-\QQ_{h}(x^1,-y^1)|=|g(x^1)-g(y^1)-\langle a_h(x^1),x^1-y^1\rangle|\\
    =&|g(x^1)-g(y^1)-\langle \nabla g(x^1),x^1-y^1\rangle|\leq |\langle \nabla g(\xi)-\nabla g(x^1),x^1-y^1\rangle|\\
    \leq& K'|x^1-y^1||\xi-x^1|(1+|\log|\xi-x^1||)^K\leq Kd(x,y)^2(1+|\log (Cd(x,y))|)^K.
\end{split}
\end{equation*}
Here $\xi$ is a suitable point in the line segment joining $x^1$ and $y^1$. The last inequality follows from \eqref{comparison between d_E and d_CC} and from the fact that the function $r\mapsto r(1+|\log r|)^K$ is increasing on an interval $(0,c_K)$, therefore it holds when $d(x,y)\leq c_K$.
Combining the two previous inequalities, we obtain
\begin{equation}\label{eq: t-g(x)-s+g(y)}
     \left|x_h-g(x^1)-(y_h-g(y^1))\right|\leq C_1 d(x,y)^2(1+|\log (C_2d(x,y))|)^K.
\end{equation}
for some $C_1,C_2>0$.
Since the function $r^{\frac{1+\alpha}{2}}e^{-M\left|\log r\right|^\beta}$ is strictly increasing for $0<r\leq1$, combining 
\eqref{eq: t-g(x)-s+g(y)} and \eqref{eq: f(p)-f(p')} we get

\begin{equation}\label{eq whitney estimate}
    \frac{|f^*(x)-f^*(y)|}{d(x,y)^{1+\alpha}}\leq C
    \frac{\left(C_1 (1+|\log (C_2d(x,y))|)^ K\right)^{\frac{1+\alpha}{2}}}{ e^{M\left|\log\left( C_1 d(x,y)^2(1+|\log (C_2d(x,y))|)^K\right)\right|^\beta}}=C\Upsilon(d(x,y)).
\end{equation}
We observe that
\begin{equation}\label{eq:0whitney}
\limsup_{\lambda\to0}\Upsilon(\lambda)\leq \limsup_{\lambda\to0}\frac{\left|\log \lambda\right|^{K{\frac{1+\alpha}{2}}}}{e^{|\log \lambda|^{\beta}}}=\lim_{\tau\to \infty}\frac{\tau^{K{\frac{1+\alpha}{2\beta}}}}{e^{\tau}}=0,
\end{equation}
where in the last step we used the substitution $\tau=|\log\lambda|^\beta$.
Thus, the boundedness of $\Upsilon$ near $0$, together with
\eqref{eq whitney estimate}, implies the Whitney's condition \eqref{whitney condition}. When $\alpha=0$, the limit \eqref{eq:0whitney} allows us to apply \cite[Theorem 5.2]{FSSC5}.
As a consequence, there exists an extension 
$\tilde{f}\in C^{1,\alpha}_{\mathrm{H}}(\G)$ of $f^*$ such that $(\nabla_{\mathrm{H}} \tilde{f})\big|_{F}\equiv 0$.
\end{proof}

\begin{remark}\label{rem:idea_sing_hyp}
Let $x\in\R^q$, and denote by $\tilde{x}\in\R^{q-1}$ the vector $x$ with the $h$-th component omitted, and by $x^1\in\R^m$ the vector whose 
first $m$ components coincide with those of $x$.
By Lemma~\ref{lemma:Balogh}, and following its notation, 
we get a $C^1$ function $G\colon Q\times \R^{q-m-1}\to\R$, 
which is constant with respect to the variables 
$(x_{m+1},\ldots,x_{h-1},x_{h+1},\ldots,x_q)$, 
such that $G(\tilde x)=g(x^1)$, and
    \begin{equation*}
        \nabla G(\tilde{x})=
        (\nabla g(x^1),0,\ldots,0)=(a_{1h}(x^1),\ldots,a_{mh}(x^1),0,\ldots,0)
    \end{equation*}
for all $\tilde x\in F_g\times\R^{q-m-1}$.
We have chosen the vector field 
$X$ of Lemma~\ref{lemma:Balogh} equal to the
one of \eqref{eq:Xah}.  
The measurable subset $F_g$ of the $m$-dimensional cube
$Q$ is such that $\mathcal L^m(Q\setminus F_g)$ can be 
chosen to be arbitrarily small. 
Let $\Gamma\colon Q\times\R^{q-m-1}\to\R^q$ denote the graph map
    \begin{equation*}
    \Gamma(\tilde{x})=(x_1,\ldots,x_{h-1}, G(\tilde{x}),x_{h+1},\ldots,x_q)
    \end{equation*}
and notice that the graph $\Sigma$ of $G$ is a $C^1$ submanifold
of $\G$. 
For $\tilde x\in F_g\times \R^{q-m-1}$ and
$i=1,\ldots,m$, we have
    \begin{align*}        
    X_i(\Gamma(\tilde x))&=\partial_{x_i}+a_{ih}(\Gamma(\tilde x))\partial_{x_h}+\sum_{\substack{j=m+1\\j\neq h}}^q a_{ij}(\Gamma(\tilde x))\partial_{x_j} \\
    &=\partial_{x_i}+\partial_iG(\tilde x)\partial_{x_h}+\sum_{\substack{j=m+1\\j\neq h}}^q a_{ij}(\Gamma(\tilde x))\partial_{x_j}=\partial_i\Gamma+\sum^q_{\substack{j=m+1\\j\neq h}} a_{ij}(\Gamma(\tilde x))\partial_j\Gamma.
    \end{align*}
This shows that for every $\tilde x\in F_g\times \R^{q-m-1}$ the horizontal vectors at $\Gamma(\tilde x)$ are contained in $T_{\Gamma(\tilde x)}\Sigma$.
In other words, $\Gamma(\tilde x)$ are characteristic points
of $\Sigma$ and by the standard area formula one easily notices that
they have positive $(q-1)$-dimensional Euclidean Hausdorff measure.
\end{remark}

\begin{proof}[Proof of Theorem~\ref{main theor codim k}]
We follow the notation in the statement of Lemma~\ref{lem: extension}. 
Let $C_g>0$ be such that $\|g\|_{\infty}\leq C_g$. Let $K\subset\R^{q-m-1}$ be an arbitrary compact set with $\L^{q-m-1}(K)>0$, define $t=(t_1,\ldots,t_{q-m})\in\R^{q-m}$ and
\[
T=\Big\{t\in\R^{q-m}: (t_1,\ldots,t_{h-m-1},t_{h-m+1},\ldots t_{q-m})\in K,\,
t_{h-m}\in [-C_g-2,C_g+2]\Big\}.
\]
Lemma~\ref{lem: extension} gives us $\tilde{f}\in C^{1,\alpha}_{\mathrm{H}}(\G)$ such that for every $x\in F_g\times T$ we have 
\begin{enumerate}[label=\normalfont{(\roman*)}]
    \item $\tilde f(x)=\psi(x_h-g(x^1))$,
    \item $\nabla_{\mathrm{H}} \tilde f(x)=0$.
\end{enumerate}
We define
$f=(f_1,\ldots,f_{k})\colon \G\to\R^k$ by
\begin{equation*}
    f_1(x)=\tilde{f}(x)-x_1\quad\text{and}\quad f_{i+1}(x)=x_{i+1}\quad\text{for }i=1,\ldots,k-1.
\end{equation*}
Clearly the h-Jacobian matrix $\nabla_\mathrm Hf(x)$ has maximal rank $k$ for every $x\in F_g\times T$, and thus in an open set $\Omega\supset F_g\times T$.
First, we assume $k>1$. By Fubini's Theorem we can find
 $a=(a_2,\ldots,a_k)\in[0,1]^{k-1}$ such that
 \[
 F_{g,a}=\{(x_1,x_{k+1},\ldots,x_m)\in[0,1]^{m-k+1}\colon
 (x_1,a_2,\ldots,a_k,x_{k+1},\ldots,x_m)\in F_g\}
 \]
has positive $m-k+1$-dimensional Lebesgue measure.
We set $a_0=(0,a_2,\ldots,a_k)\in\R^k$ and define  
$S=\Omega\cap f^{-1}(a_0)$.
We notice that the following inclusion holds
\begin{equation*}
S\supset\{x\in F_g\times T\colon x_h\in \psi^{-1}(x_1)+g(x^1)\text{ and }x_{i+1}=a_{i+1} \text{ for }i=1,\ldots,k-1\}=S'_a.
\end{equation*}
Taking into account the discussion about $\psi$ below Lemma~\ref{lemma:Rajala-alpha}, we have
\[
\psi^{-1}(z_1)+g(z)\subset [-2-C_g,C_g+2]
\]
for every $z\in Q$. It follows that $S'_a\neq\emptyset$, and therefore $S$ is nonempty. 
By definition of $\Omega$, we notice that $S$ is also a co-horizontal surface. Let $\pi_k\colon \R^q\to\R^{q-k}$  be projection that removes the components $x_2,x_3,\ldots,x_k,x_h$. Thus, we obtain 
\[
\pi_k(S)\supset\pi_k(S'_a)=F_{g,a}\times K.
\]

Now, we arbitrarily fix $0<\eps<1-\frac{1+\alpha}{2}$. 
Due to the Eilenberg inequality (Lemma~\ref{lemma:Eilenberg}),
there exists $c_\eps>0$ such that 
\begin{align*}
c_\eps \mathcal{H}^{q-k+1-\frac{1+\alpha}{2}-\eps}&(S)\geq \int_{F_{g,a}\times K} \mathcal{H}^{1-\frac{1+\alpha}{2}-\eps}(\pi_k^{-1}(y)\cap S'_a)d\L^{q-k}(y)\\
    =&\int_{F_{g,a}\times K} \mathcal{H}^{1-\frac{1+\alpha}{2}-\eps}\left(\psi^{-1}(y_1)+g(y_1,a,y_{k+1},\ldots,y_m)\right)d\L^{q-k}(y)\\   =&\int_{K}\left(\int_{F_{g,a}} \mathcal{H}^{1-\frac{1+\alpha}{2}-\eps}\left(\psi^{-1}(z_1)\right)d\mathcal L^{m-k+1}(z)\right)d\L^{q-m-1}=+\infty,
\end{align*}
where the last equality holds since  $\dim_\mathrm E(\psi^{-1}(t)) \ge 1 - \frac{1+\alpha}{2}$ for every $t\in \R$ by (i) of Lemma~\ref{lemma:Rajala-alpha}, as well as $\mathcal L^{m-k+1}(F_{g,a})>0$ and $\L^{q-m-1}(K)>0$. Therefore, from the arbitrariness of $\eps$, we conclude that $\dim_\mathrm ES\geq q-k+\frac{1-\alpha}{2}$.

For $k=1$, the argument above holds and simplifies by taking $F_{g}$ in place of $F_{g,a}$.
\end{proof}

\section{The case of positively graded groups}\label{sec:graded}

In this section, we extend some considerations about our 
problem in the larger class of positively graded nilpotent Lie groups. They differ from Carnot groups since the only requirement on the
Lie algebra is the direct decomposition
\begin{equation*}
\Lie(\G)=V_1\oplus V_2\oplus\cdots\oplus V_s,
\end{equation*}
with the grading property
\begin{equation*}\label{eq: graded inclusion}
    [V_i,V_j]\subset V_{i+j}\quad \text{ for all $i,j\ge1$,}
\end{equation*}
with $V_j=\{0\}$ whenever $j>s$.
The important difference with respect to 
Carnot groups is that we may have $\QQ_h\equiv 0$ for some index $h>m$,
as proved in Proposition~\ref{prop: not stratified iff Rj=0}. 
Thus, the construction of an intrinsic regular hypersurface $S$ that is irregular in the standard sense does not rely on the results by Alberti and Balogh, \cite{Alberti91,Balogh03}.
Indeed, a purely Euclidean construction suffices, as it will be clear from the proof of the next proposition.

\begin{prop}\label{prop: graded groups construction}
Let $\G$ be a positively graded group which is not stratified and let $0<\alpha<1$. Then there exists $f\colon\G\to\R$ which is a locally $\alpha$-H\"older continuous function with respect to the Euclidean distance on $\G$ and such that $X_if:\G\to\R$ is smooth for each $i=1,\ldots,m$, the horizontal gradient $\nabla_{\mathrm{H}}f$ is everywhere nonvanishing, and the level set $S=f^{-1}(0)\subset\G$ satisfies $\HH^{q-\alpha}(S)>0.$
\end{prop}

\begin{proof}
By Proposition~\ref{prop: not stratified iff Rj=0}, we can fix a graded basis on $\G$ so that $R_h\equiv 0$ for some $h>m$.
By virtue of \cite[Proposition~2.4]{Rajala2026pr}, there exists $\psi\colon\R\to\R$, which is $\alpha$-H\"older continuous and all its level sets have positive
$(1-\alpha)$-dimensional Hausdorff measure.
We define the function
    \begin{equation*}
        \tilde{f}\colon\G\to\R,\quad x\mapsto\psi(x_h),
    \end{equation*}
where $x_h$ is the $h$-th component of $x$ with respect to the fixed graded basis.
Then, recalling \eqref{eq: def of Xi wrt base wi} and \eqref{eq: def of bij}, the condition $R_h\equiv0$ implies that $b_{ih}\equiv 0$ and therefore $X_i\tilde f(x)=0$ for every $x\in\G$ and for every $i=1,\ldots,m$.

Let $f(x)= \tilde{f}(x)-x_1$, therefore $X_1f\equiv-1$. 
Then we set $S=f^{-1}(0)$ and consider the projection 
$\pi\colon \R^q\to\R^{q-1}$ that removes the component $x_h$. Then, by definition of $f$, for every $z\in\R^{q-1}$ we have
\begin{equation*}
    S \cap \pi^{-1}(z)=\{x\in\G\colon x_h\in\psi^{-1}(z_1) \text{ and } x_i=z_i\text{ for every } i\neq h\}.
\end{equation*}
Thus, by the Eilenberg inequality, there exists $c>0$ such that
\begin{align*}
    c\; \mathcal{H}^{q-\alpha}(S)\geq& \int_{\R^{q-1}} \mathcal{H}^{1-\alpha}(S\cap\pi^{-1}(z))d\HH^{q-1}(z)\\
    &=\int_{\R^{q-1}} \mathcal{H}^{1-\alpha}(\psi^{-1}(z_1))d\HH^{q-1}(z)>0,
\end{align*}
thereby concluding the proof.
\end{proof}

The proposition shows that for non-stratified groups,
the Hausdorff dimension can be arbitrarily close to $q$
when $\alpha>0$ is small.
Proposition~\ref{prop: graded groups construction} also holds in the case $\alpha=0$, where $f$ is continuous and $\dim_ES=q$. The proof is achieved by the same construction, 
where $\psi\colon\R\to\R$ is a continuous function having almost every level set of Euclidean Hausdorff dimension $1$,
see for instance \cite[Theorem 6.1]{BalDarEle2016}.  

\section{Upper bound for the Hausdorff dimension}

In this section $\G$ denotes a step-2 Carnot group of topological dimension $q$ and $S\subset\G$ is a $C^{1,\alpha}_{\mathrm{H}}$
hypersurface, if not otherwise stated. 
We aim to show that the Euclidean Hausdorff dimension of $S$ 
cannot be larger than $q-\frac{1+\alpha}{2}.$ As a result, 
the hypersurface exhibited in Theorem~\ref{main theor codim k}
achieves the maximal Euclidean Hausdorff dimension. 

\begin{remark}\label{rem: btw estimate}
We observe that if $\G$ is a step $s$ Carnot group and $S$ is $C^{1,\alpha}_{\mathrm{H}}$, general comparison estimates between the Euclidean and the homogeneous Hausdorff dimensions \cite[Theorem 2.4]{BTW09} yield the following upper bound
\begin{equation}\label{eq: upper bound BTW}
    \dim_\mathrm ES\leq q-\frac{1}{s}.
\end{equation}
In particular, when $s=2$ this covers Theorem~\ref{main theorem upper estimate} for $\alpha=0$. 
\end{remark}
Throughout the section, for $x\in\G$ we follow the notation 
$$
x=(x^1,x^2)\in\R^m\times\R^{q-m},
$$
where $x^1$ and $x^2$ denote the horizontal and vertical components of $x$, respectively.

In the following proposition we use the notion of locally $C^{1,\alpha}_{\mathrm H}$ function.
We say that a function $f:\Omega\to\R$ is $C^{1,\alpha}_{\mathrm{H},\text{loc}}$,
or that it belongs to $C^{1,\alpha}_{\mathrm{H},\text{loc}}(\Omega)$,
where $\Omega\subset\G$ is open, if $f\in C^1_{\mathrm H}(\Omega)$ and for every compact $K\subset\Omega$ there exists a constant $C_K>0$ such that 
\[
|\nabla_Hf(x)-\nabla_Hf(y)|\leq C_Kd(x,y)^\alpha\quad\text{for every $x,y\in K$.}
\]
\begin{prop}\label{prop: from C^1,alpha intrinsic to Euclidean Holder}
Let $\G$ be a step-$2$ Carnot group, let $\Omega\subset\G$ be open, and fix $\alpha\in[0,1]$. Then every $f\in C^{1,\alpha}_{\mathrm{H},\text{loc}}(\Omega)$ is also locally $\frac{1+\alpha}{2}$-H\"older continuous in the Euclidean sense.
\end{prop}
The previous proposition is a special consequence 
of \cite[Theorem~5.25]{Fol75}.

\begin{remark}[Mean value theorem]\label{lagrange lemma}
Let $\gamma:[a,b]\to\G$ be a $C^1$ horizontal curve
and let $f\in C^1_{\mathrm{H}}(\Omega)$, where $\Omega\subset\G$ 
is an open set containing $\gamma([a,b])$. 
To obtain a mean value theorem for $f\circ\gamma$, 
the standard chain rule is not enough, since $f$ need not be
$C^1$ smooth in the Euclidean sense, as shown in \cite[Section~4]{Mag6}.
For this reason, the chain rule with respect to
the Pansu differentiability is necessary;
see for instance \cite[Proposition~2.11]{Mag10}. 
Then the differentiability of 
$t\to f(\gamma(t))$ follows from \cite[Theorem~1.1]{Mag14},
and finally the standard Lagrange theorem yields
the existence of $\xi\in(a,b)$ such that
\begin{equation}\label{eq: lagrange mean value}
    f(\gamma(b))-f(\gamma(a))=(b-a)\langle\nabla_{\mathrm{H}}f(\gamma(\xi)),\dot{\gamma}(\xi)\rangle.
\end{equation}
\end{remark}
In view of the proof of the next theorem, we recall that 
in a step-$2$ Carnot group the polynomials $\QQ_i$ arising from the BCH formula, are bilinear and antisymmetric and are written as
\begin{equation}\label{eq: qi in step 2}
    \QQ_i(x^1,y^1)=\frac{1}{2}[x^1,y^1]_i=\sum_{j,k=1}^m b^i_{jk}x_jy_k, \quad\quad b^i_{jk}=-b^i_{kj},
\end{equation}
for every $m+1\leq i\leq q$.

In the sequel, the subset $B(x,r)\subset\G$ will denote the open metric ball of center $x$ and radius $r>0$
with respect to the homogeneous distance $d$ of \eqref{def of d}.

\begin{theorem}[$\frac{1+\alpha}{2}$-H\"older continuity]\label{intrinsic parametrization is Holder}
Let $\G$ be a Carnot group of step $2$, let $\Omega\subset\G$ be open, let $f\in C^{1,\alpha}_{\mathrm{H},{\rm loc}}(\Omega)$ and assume that 
$\nabla_Hf(x_0)\neq0$ for some $x_0\in f^{-1}(f(x_0))=S$.
We fix a graded basis $(e_1,\ldots,e_q)$ of $\G$ 
with associated basis $(X_1,\ldots,X_q)$ of left-invariant vector fields in $\Lie(\G)$.
We choose $j\in\{1,\ldots,m\}$ for which 
$X_jf(x_0)\neq0$.
Let us define the homogeneous subgroups 
\[
\V=\spann\{e_j\}\quad \text{and} \quad \W=\spann\{e_1,\ldots,e_{j-1},e_{j+1},\ldots,e_q\}.
\]
Then there exist $r_0>0$, a bounded open neighborhood $U\subset\W$ of $\pi_{\W}(x_0)$, and  
$\phi\colon U\to\V$ such that
$S\cap B(x_0,r_0)=\graph{\phi}$ and 
$\phi$ is $\frac{1+\alpha}2$-H\"older continuous
with respect to the Euclidean distance.
\end{theorem}

\begin{proof}
Without loss of generality, we may assume that $j=1$
and fix a system of graded coordinates associated with
$e_1,\ldots,e_q\in\G$. 
From Theorem~\ref{implicit function theorem},
we can find $r>0$, an open neighborhood $\widetilde U\subset\W$ of $\pi_\W(x_0)$, and a continuous 
function $\phi:\widetilde U\to\V$ such that 
$S\cap B(x_0,r)=\graph{\phi}\big|_{\widetilde U}$,
with $B(x_0,r)\subset \Omega$.
We can also assume that $X_1f$ is everywhere nonvanishing
on $B(x_0,r)$.
Let $x,y\in S\cap B(x_0,\frac{r}{3})$, therefore in our coordinates 
\[
x=w (\phi(w),0,\ldots,0)\quad\text{and}\quad 
y=\tilde{w} (\phi(\tilde{w}),0,\ldots,0)
\]
for some $w,\tilde w\in \tilde{U}$. We define
    \begin{equation*}
            v=x^{-1}y
    \end{equation*}
and observe that $v=(v_1,v_2,\ldots,v_q)$, where
\begin{equation}\label{eq v_i}
    v_i= 
    \begin{cases}
        \phi(\tilde{w})-\phi(w) & i=1,\\
        \tilde{w}_i-w_i & 2\leq i\leq m,\\
        (w^{-1}\tilde{w})_i+(\phi(\tilde{w})+\phi(w))\sum_{j=2}^m b^i_{j1}(\tilde{w}_j-w_j)& m+1\leq i\leq q.
    \end{cases}
\end{equation}
Indeed, for $m+1\leq i\leq q$, using \eqref{eq: qi in step 2} we compute
\begin{equation*}
    \begin{split}
        v_i&=y_i-x_i-\sum_{j,k=1}^m b^i_{jk} x_jy_k=\tilde{w}_i+\sum_{j=2}^m b^i_{j1}\phi(\tilde{w})\tilde{w}_j-w_i- \sum_{j=2}^m b^i_{j1}\phi(w)w_j\\
        &-\sum_{k=2}^m b^i_{1k}\phi(w)\tilde{w}_k-\sum_{j=2}^m b^i_{j1}\phi(\tilde{w})w_j-\sum_{j,k=2}^m b^i_{jk}w_j\tilde{w}_k\\
        &=\left( w^{-1}\tilde{w} \right)_i+(\phi(\tilde{w})+\phi(w))\sum_{j=2}^m b^i_{j1}(\tilde{w}_j-w_j).
    \end{split}
\end{equation*}
Following our notation, $v^1=(x^{-1}y)^1=(v_1,\ldots,v_{m})$, then we consider
    \begin{equation*}
        \gamma_{x,v}(t)=x(tv_1,\ldots,tv_{m},0,\ldots,0)\quad\quad t\in[0,1].
    \end{equation*}
We notice that
    \begin{equation*}
     d\big(x_0,\gamma_{x,v}(t)\big)\leq
     d(x_0,x)+\|\big(t(x^{-1}y)^1,0\big)\|_\G<\frac{r}{3} +\|\big((x^{-1}y)^1,0\big)\|_\G<r,
    \end{equation*}
    hence $\gamma_{x,v}(t)\in B(x_0,r)\subset\Omega$ for every $t\in[0,1]$.
    By Remark~\ref{lagrange lemma}, we get $t_0\in[0,1]$ with
    \begin{equation}\label{eq f(p')-f(p)}
    \begin{split}
        0&=f(y)-f(x)=f(y)-f(\gamma_{x,v}(1))+f(\gamma_{x,v}(1))-f(\gamma_{x,v}(0))\\
        &=f(y)-f(x v^1)+\sum_{i=1}^m X_i f(\gamma_{x,v}(t_0))v_i.
    \end{split}
    \end{equation}
Then, we isolate the term containing $v_1$ in \eqref{eq f(p')-f(p)}, recalling that $v_1=\phi(\tilde w)-\phi(w)$. By \eqref{eq v_i} and Proposition~\ref{prop: from C^1,alpha intrinsic to Euclidean Holder}, there exist $M_1,M_2>0$ such that
\begin{equation}\label{eq phi(w')-phi(w)}
    \begin{split}
        |\phi(\tilde{w})-\phi(w)|
        &\leq\frac{1}{|X_1 f(\gamma_{x,v}(t_0))|}\left(|f(y)-f(x v^1)|+\sum _{i=2}^m |X_i f(\gamma_{x,v}(t_0))||\tilde{w}_i-w_i| \right)\\
        &\leq M_1|y-x v^1|^{\frac{1+\alpha}{2}}+M_2 |\tilde{w}-w|.
    \end{split}
\end{equation}
Moreover, recalling the definition of $v$ and using \eqref{eq v_i}, \eqref{comparison between d_E and d_CC}, and the boundedness of $\phi$, there exist $C,C_1,C_2>0$ such that
\begin{equation*}
\begin{split}
    |y-x v^1|&=|x v-x v^1|=|x v^1 v^2-x v^1|=|v^2|\\
    &\leq\sum_{i=m+1}^q |v_i|\leq C_1 d(w,\tilde{w})^2+C_2\sum_{j=2}^m|\tilde{w}_j-w_j|\\
    &\leq C|w-\tilde{w}|.
\end{split}
\end{equation*} Substituting this estimate back into
\eqref{eq phi(w')-phi(w)}, we conclude that there exists $L>0$ such that
\begin{equation*}
    |\phi(\tilde{w})-\phi(w)|\leq L|w-\tilde{w}|^{\frac{1+\alpha}{2}},
\end{equation*}
where $L$ is independent of $x$ and $y$.
Setting $w\to\Phi(w)=w\phi(w)$, 
$U=\Phi^{-1}\big(B(x_0,r/3)\big)$ and $r_0=r/3$, the proof is complete.
\end{proof}

\begin{cor}\label{cor alpha=1}
    Let $\G$ be a Carnot group of step $2$
    and let $S\subset\G$ be a $C^{1,1}_{\mathrm{H}}$ regular hypersurface. 
    Then $S$ is $(q-1)$-Euclidean rectifiable.
\end{cor}
\begin{proof}
We fix $x\in S$ and apply Theorem \ref{intrinsic parametrization is Holder} to this point. Thus, we can find two complementary subgroups $\W$ and $\V$ as in the statement of the theorem, an open ball $B(x,r_x)\subset\G$, an open neighborhood $U_{x}\subset\W$ of $\pi_\W(x)\in\W$, and a Lipschitz continuous function $\phi \colon U_{x}\to\V$ such that $S\cap B(x,r_x)=\graph\phi$. The graph mapping
\[ 
\Phi\colon U_{x}\to S\cap B(x,r_x),\quad
            w\mapsto w\phi(w),
\]
is also Lipschitz continuous with respect to the Euclidean distance. Thus, the local piece $S\cap B(x,r_x)=\Phi(U_x)$ is $(q-1)$-Euclidean rectifiable.
Since $\G$ is second countable, we may extract a countable subcovering of $S$ from the covering $\{B(x,r_x)\colon x\in S\}$, hence concluding the proof.
\end{proof}

We now restate and prove Theorem~\ref{main theorem upper estimate}.
\begin{theorem}\label{thm:tau_step2}
Let $\G$ be a step-$2$ Carnot group and 
let $S\subset \G$ be a $C^{1,\alpha}_{\mathrm{H}}$ hypersurface, with $\alpha\in[0,1]$. Then we have $\dim_\mathrm E S\leq q-\frac{1+\alpha}{2}$.
\end{theorem}
\begin{proof}
As already mentioned in the introduction, for $\alpha=0$ our claim follows from the general estimates of \cite[Theorem~2.4]{BTW09}, see also Remark \ref{rem: btw estimate}.
For $\alpha=1$, it suffices to apply Corollary \ref{cor alpha=1}.
We now focus our attention on the case $0<\alpha<1$.

Fix $x_0\in S$. Let $f:\Omega\to\R$ be the $C^{1,\alpha}_{\mathrm H}$ function
such that $x_0\in\Omega\cap S=f^{-1}(0)$ and $\nabla_Hf(z)\neq0$ 
for every $z\in\Omega$. It is not restrictive to assume 
for instance that $X_1f$ is everywhere positive on $\Omega$. By Theorem~\ref{intrinsic parametrization is Holder}, there exists an open ball $B(x_0,r_0)$, 
two complementary subgroups $\W\subset\G$ and $\V\subset\G$, a bounded open set $U\subset\W$, and a Euclidean $\frac{1+\alpha}{2}$-H\"older continuous function $\phi\colon U\to \V$ such that $S\cap B(x_0,r_0)=\graph\phi$.  
There also exists a graded basis $e_1,e_2,\ldots,e_q$ of $\G$, such that 
$\W=\spann\{e_2,\ldots,e_q\}$ and $\V=\spann\{e_1\}$.
We consider the following map
\[
\Psi\colon  (0,1)\times U\to \mathcal O\subset\G, \quad 
(t,w)\mapsto w \phi(w) \exp(tX_1)
\]
and define the family of surfaces $(S_t)_{t\in (0,1)}$ as
\begin{equation*}
    S_t\coloneqq \Psi\left( \{t\}\times U\right).
\end{equation*}
We claim that $\Psi$ is invertible on its whole 
image $\mathcal O$, and that the following map is $\frac{1+\alpha}{2}$-H\"older continuous with respect to the Euclidean metric:
\begin{equation*}
    \tau\colon \mathcal O\to\R,\quad \tau(x)\coloneqq \pi_1(\Psi^{-1}(x)).
\end{equation*}
We will isometrically identify $\W$ and $\V$ with $\R^{q-1}$ and $\R$, respectively. 
We choose any $x\in \mathcal O$ and consider the system 
\begin{equation*}
\begin{split}
        x=\Psi(w,t)&\iff
        \begin{cases}
        x_1=t+\phi(w) & i=1\\
        x_i=w_i & 2\leq i\leq m\\
        x_i=w_i+\sum_{j=2}^m b^i_{j1} (t+\phi(w))w_j \quad&m+1\leq i\leq q
    \end{cases}\\
    &\iff
    \begin{cases}
       t=x_1-\tilde{\phi}(x) & i=1\\
        w_i=x_i & 2\leq i\leq m\\
        w_i=x_i-\sum_{j=2}^m b^i_{j1} x_1x_j \hspace{2cm}& m+1\leq i\leq q.
    \end{cases}
\end{split}
\end{equation*}
where $\tilde{\phi}(x)=\phi(x_2,\ldots,x_m,x_{m+1}-\sum_{j=1}^m b^{m+1}_{j1}x_1x_j,\ldots,x_{q}-\sum_{j=1}^m b^{q}_{j1}x_1x_j)$.
By the definition of $\tau$, we have
\begin{equation*}
    \tau(x)=x_1-\tilde{\phi}(x),
\end{equation*}
and it is $\frac{1+\alpha}{2}$-H\"older continuous, since $\tilde{\phi}$ is the composition of a $\frac{1+\alpha}{2}$-H\"older continuous function and a polynomial on a bounded set.
Taking into account Remark~\ref{rem: dim holder level set},
for a.e.\ $t\in(0,1)$ we get 
\begin{equation}\label{eq:dimest}
    \HH^{q-\frac{1+\alpha}{2}}(S_t)=\HH^{q-\frac{1+\alpha}{2}}( \tau^{-1}(t))<+\infty .
\end{equation}
Since the right multiplication is an analytic 
diffeomorphism of $\G$, all $S_t$ are locally bi-Lipschitz images of $S$. We conclude that \eqref{eq:dimest} also holds for $S$ in place of $S_t$, hence $\dim_E S_t\leq q-\frac{1+\alpha}{2}$.
\end{proof}

\begin{remark}
From \eqref{eq:dimest}, under the hypothesis of Theorem~\ref{thm:tau_step2}, we actually deduce the stronger statement that restriction
$\HH^{q-\frac{1+\alpha}{2}}\text{\small $\lfloor$}\,S$
is locally finite.
\end{remark}

\begin{cor}\label{cor:Holdermanifold}
Let $0\le \alpha\le 1$. Then any $C^{1,\alpha}_{\mathrm H}$ hypersurface embedded in a step-2 Carnot group is also $\frac{1+\alpha}{2}$-H\"older manifold. 
\end{cor}
\begin{proof}
We follow the notation and definitions as in the proof of 
Theorem~\ref{thm:tau_step2}. Let \[
\Psi\colon  (-1,1)\times U\to \mathcal O\subset\G, \quad 
(t,w)\mapsto w \phi(w) \exp(tX_1).
\]
Using the same argument of Theorem \ref{thm:tau_step2}, it is not difficult
to realize that $\Psi$ is locally $\frac{1+\alpha}{2}$-bi-H\"older continuous, hence so is its restriction
\[
\Psi|_{\{0\}\times U}:\{0\}\times U\to S\cap B(x_0,r_0).
\]
This immediately leads us to the conclusion.
\end{proof}

Theorem~\ref{thm:tau_step2} in the case $\alpha=0$ 
has an interesting stronger version. 
In fact, the graphing function $\phi$ defining $S\cap\Omega$ is locally $\frac12$-little H\"older continuous, namely
for every $\Omega'\Subset\Omega$ we have
\begin{equation}\label{eq:1/2littleHolder}
        \lim_{r\to 0}\;\sup_{\substack{x,y\in\Omega'\\0<|x-y|\leq r}}\frac{|\phi(x)-\phi(y)|}{|x-y|^\frac{1}{2}}=0.
\end{equation}
In step-2 Carnot groups, different proofs of this result are
available. It follows for instance from 
\cite[Theorem 5.8]{DiD21}, or from \cite[Proposition~3.21 and Theorem~1.6]{ADDDLD24}.
Applying this additional regularity \eqref{eq:1/2littleHolder} we obtain the 
following.

\begin{cor}\label{cor:Hq-1/2=0}
If $\G$ is a step-$2$ Carnot group and $S\subset \G$ is a $C^1_{\mathrm{H}}$ hypersurface, then $\HH^{q-\frac{1}{2}}(S)=0$.
\end{cor}
\begin{proof}
The implicit function theorem and \eqref{eq:1/2littleHolder} give open neighborhoods
 $\Omega\subset\G$ and $U\subset\W$ such that 
 $ S\cap \Omega=\graph \phi$
and $\phi\colon U\to\R$ is $\frac12$-little H\"older continuous.
 We can repeat the construction of the function $\tau\colon \mathcal O\to\R$ as in the proof of Theorem~\ref{thm:tau_step2}
 and following the corresponding notation. We obtain that $S_t=\tau^{-1}(t)$ for all $t\in\R$, where $\tau$ inherits the regularity of $\phi$, that is $\tau$ is also 
 $\frac12$-little H\"older continuous.
    Therefore, Lemma \ref{lem: holder+ level set} implies that
    \begin{equation}\label{eq:Hq1/2=0}
    \HH^{q-\frac{1}{2}}(S_t)=0
    \end{equation}for almost every $t\in(0,1)$.
    By the construction of Theorem \ref{thm:tau_step2}, 
    all $S_t$ are locally bi-Lipschitz images of $S$,
    hence \eqref{eq:Hq1/2=0} holds also for $S$.
\end{proof}

\begin{remark}\label{rem:littleHoldermanifold}
Arguing as in the proof of Corollary~\ref{cor:Holdermanifold} and 
taking into account the regularity result \eqref{eq:1/2littleHolder}, we can 
also conclude that in a step-2 Carnot group any $C^1_{\mathrm H}$ hypersurface is
also a $\frac12$-little H\"older manifold.
\end{remark}

\section{H\"older functions with irregular level sets}\label{section: app}

In this section, we present Rajala's construction of special classes
of H\"older continuous functions of one variable.
Other related results on H\"older continuous functions of one variable are available, \cite{AnttilaBaranyKaenmaki2025}.

In \cite{Rajala2026pr} Rajala constructs an $\alpha$-H\"older continuous function from $\R$ to $\R$ whose level sets all have positive $(1-\alpha)$-dimensional Hausdorff measure. 
Such a function exhibits maximal irregularity, as $1-\alpha$ is the highest Hausdorff dimension that a.e.\ level set can attain.

A similar result can be established for functions with $\alpha$-little H\"older regularity, only requiring that the Hausdorff dimension of all their level sets is at least $1-\alpha$.

\begin{theorem}\label{thm: rajala holder+}
    For every $\alpha,\beta\in(0,1)$, there exists a function $\psi\colon[0,1]\to[0,1]$ such that $\psi(0)=0$, $\psi(1)=1$ and
    \begin{enumerate}[label=\emph{(\roman*)}]
        \item $\dim_\mathrm E\psi^{-1}(t)\geq 1-\alpha$ for all $t\in[0,1]$;
        \vspace{4pt}
        \item for every $M>0$ we have
        \begin{equation*}
            \lim_{r\to0^+}\frac{e^{M|\log r|^\beta}}{r^{\alpha}}\sup\{|\psi(x)-\psi(y)|\colon|x-y|\leq r\}=0.
        \end{equation*}
    \end{enumerate}
\end{theorem}
For $\alpha=1/2$ the previous result recovers \cite[Lemma~3.5]{KirSer04} and \cite[Proposition~2.6]{Rajala2026pr}. The main ideas underlying the proof of Theorem~\ref{thm: rajala holder+} are already present in \cite{Rajala2026pr}. We adapt that approach to every 
H\"older exponent of continuity $0<\alpha<1$.

We observe that the function $\psi$ in Theorem~\ref{thm: rajala holder+} is slightly more regular than $\alpha$-H\"older continuous, but it cannot be $\sigma$-H\"older for any $\alpha<\sigma\le 1$, 
due to Remark~\ref{rem: dim holder level set}.
Moreover, the modulus of continuity $e^{-M|\log r|^\beta}r^\alpha$ decays faster than $|\log r|^{-K}r^\alpha$ for any $K>0$, hence $\psi$ in  Theorem~\ref{thm: rajala holder+} gains slightly more regularity than that in \cite[Lemma 3.5]{KirSer04}.
We observe that in all of these cases the modulus of continuity implies $\HH^{1-\alpha}(\psi^{-1}(t))=0$ for almost every $t\in\R$, see Lemma~\ref{lem: holder+ level set}.

In the following subsections~\ref{subs: construction of g} and~\ref{subs: frostman}, we describe the ideas of Rajala's construction, which will be used in the proof of Theorem~\ref{thm: rajala holder+}.

\subsection{Construction of the irregular Hölder function.}\label{subs: construction of g}

We fix the square $Q = [0,1]^2$. For each $m \in \mathbb{N}^+$ and $i \in \{0,\ldots,3m-1\}$, we define
\begin{equation*}
f_i^m(x,y)
\coloneqq  \left(
  \frac{x+i}{3m},
  (-1)^{\left\lfloor \frac{i}{m} \right\rfloor}
  \frac{y + i - \left\lfloor \frac{i}{m} \right\rfloor m}{m}
  + \frac{1 + (-1)^{\left\lfloor \frac{i}{m} \right\rfloor + 1}}{2}
  \right),
\end{equation*}
where $\lfloor x\rfloor=\max\{k\in\Z:k\le x\}$ for every $x\in\R$.

Each $f_i^m$ maps $Q$ into a rectangle $R_i \coloneqq  f_i^m(Q)\subseteq Q$, with base $1/(3m)$ and height $1/m$.  
These rectangles are arranged in an ``$N$-shaped'' pattern. Indeed, 
plotting the $3m$ rectangles produces three vertical columns across $Q$. In the first and third columns ($0 \le i \le m-1$ and $2m \le i \le 3m-1$) the rectangles are arranged from bottom to top as $i$ increases, while in the middle column ($m \le i \le 2m-1$) they are arranged from top to bottom, see Figure~\ref{fig:rajala1}. 

\begin{figure}[ht]
    \centering        \includegraphics[width=0.66\linewidth]{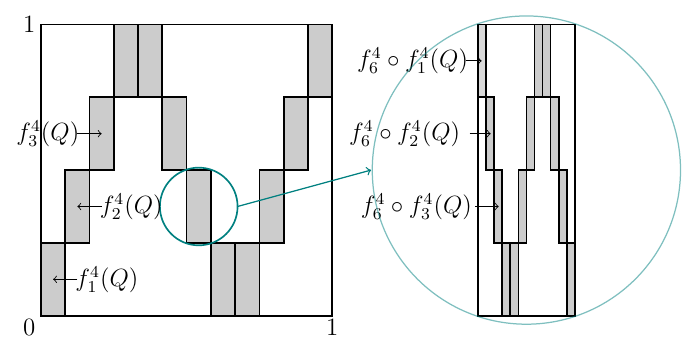}
    \caption{On the left, the maps $(f^4_i)_{i=0}^{11}$ produce an $N$-shaped pattern inside $[0,1]^2$. On the right, the maps $(f^4_6\circ f^4_i)_{i=0}^{11}$ produce a reversed $N$-shaped pattern inside $f^4_6(Q)$.}
        \label{fig:rajala1}
\end{figure}

For each $m\in\N$ with $m\geq2$, we define a map on compact subsets of $Q$ by
\begin{equation*}
F_m(A) \coloneqq  \bigcup_{i=0}^{3m-1} f_i^m(A).
\end{equation*}
Each $f_i^m$ is a contraction on $Q$; hence $F_m$ is a contraction on the space of nonempty compact subsets of $Q$ endowed with the Hausdorff distance.  
By the Banach fixed-point theorem, $F_m$ admits a unique fixed compact set, called its \emph{attractor}, given by
\begin{equation}\label{def of Gm}
G_m \coloneqq  \bigcap_{k=1}^\infty F_m^{\,k}(Q)
     = \bigcap_{k=1}^\infty
       \underbrace{F_m \circ \cdots \circ F_m}_{k\ \text{times}}(Q),
\end{equation}
see \cite[Theorem 9.1]{Falconer2003} for a reference.
Then, we can prove that $G_m$ is the graph of a $\frac{\log m}{\log (3m)}$-H\"older function whose level sets have positive $\left(1-\frac{\log m}{\log (3m)}\right)$-dimensional Hausdorff measure. The proof will be given in a more general form below.

For the general case, fix $\alpha\in(0,1)$ and let $m\in\N^+$ be the unique integer such that 
    \begin{equation*}
        \frac{\log m}{\log (3m)}\leq \alpha < \frac{\log (m+1)}{\log (3(m+1))}.
    \end{equation*}
    To get the H\"older exponent $\alpha$, we consider the composition of both functions $F_m$ and $F_{m+1}$. For every $k\in\N$, we define $G_0=Q$ and
    \begin{equation*}    G_k\coloneqq  F_{j_1}\circ\ldots\circ F_{j_k}(Q)\quad\text{and}\quad G\coloneqq \bigcap_{k=0}^\infty G_k,
    \end{equation*}
    where for all $l\in\N^+$ we set
     \begin{equation*}
    j_l=
        \begin{cases}
            m&\quad\text{if }\varphi(l)=\varphi(l-1),\\
            m+1&\quad \text{if }\varphi(l)=\varphi(l-1)+1.
        \end{cases}
    \end{equation*}
    and $\varphi\colon \N\to\N$ is a nondecreasing function satisfying
    \begin{equation*}
        \varphi(0)=0 \quad\text{and}\quad\varphi(l)\leq\varphi(l-1)+1. 
    \end{equation*}
    When $m=1$, we also assume that $\varphi(l)\to+\infty$ as $l\to+\infty$.
We notice that for each $k\ge1$ we have 
\begin{equation*}
\varphi(k)=\sum_{l=1}^k (\varphi(l)-\varphi(l-1))=\text{card}\{l\colon j_l=m+1,\,1\leq l\leq k\}.
\end{equation*}
In other words, after $k$ iterations, we have used $\varphi(k)$ times the function $F_{m+1}$ and $k-\varphi(k)$ times the function $F_m$.
Thus, $G_k$ consists of $(3m)^{k-\varphi(k)}(3(m+1))^{\varphi(k)}$ rectangles of the form
\begin{equation}\label{eq: def of Rk}
    R^k\coloneqq f_{i_1}^{j_1} \circ \cdots \circ f_{i_k}^{j_k}(Q),
\end{equation}
for every $1\leq l\leq k$ and $i_l\in\{0,\ldots,3j_l-1\}$.
These level-$k$ rectangles have base and height, respectively, of length
\begin{equation*}
    b_k\coloneqq (3m)^{-k+\varphi(k)}(3(m+1))^{-\varphi(k)}\quad\text{and}\quad h_k\coloneqq m^{-k+\varphi(k)}(m+1)^{-\varphi(k)}.
\end{equation*}
For every $k\geq1$ the family of level-$(k+1)$ rectangles is nested within the family of the rectangles of level $k$. Indeed, 
\begin{equation*}
f_{i_1}^{j_1} \circ \cdots \circ f_{i_k}^{j_{k}}\circ f_{i_{k+1}}^{j_{k+1}}(Q)
   \subseteq f_{i_1}^{j_1} \circ \cdots \circ f_{i_k}^{j_k}(Q).
\end{equation*}

For every $x_0\in[0,1]$, the set $G\cap\{x=x_0\}$ is the intersection of all nested sets $G_k\cap\{x=x_0\}
$ for $k\in\N^+$, which are compact nonempty intervals of length less than or equal to $2h_k$, according to Lemma~\ref{lem:Tapio graph} below. Since $h_k\to 0$ as $k\to+\infty$, the set $G\cap\{x=x_0\}
$ is a singleton, thus $G$ is the graph of a function $\psi\colon[0,1]\to[0,1]$.

\begin{lemma}\label{lem:Tapio graph}
    For every $k\in\N$ and $x_0\in[0,1]$, the line $x=x_0$ intersects $G_k$ in a single or in two consecutive level-$k$ rectangles, in one of the configurations A,B,C or D of Figure~\ref{fig:tapio2}.
\end{lemma}

\begin{proof}
First, by induction on $k$,  for each level-$k$ rectangle $R^k$ we define its positive or negative \emph{orientation}.
For $k=0$ we say that $R^0=Q$ is positively oriented.
For $k\geq0 $, a level-($k+1$) rectangle $R= f_{i_1}^{j_1} \circ \cdots \circ f_{i_{k+1}}^{j_{k+1}}(Q)$ has the same orientation as its predecessor $f_{i_1}^{j_1} \circ \cdots \circ f_{i_k}^{j_k}(Q)$ if  $0 \le i_{k+1} \le m-1$ or $2m \le i_{k+1} \le 3m-1$, and the opposite orientation if $m \le i_{k+1} \le 2m-1$.    
The orientation is connected to the arrangement of the rectangles. If the level-$k$ rectangle $R^k$ is positively oriented, then the family of level-$(k+1)$ rectangles contained in $R^k$ forms an $N$-shaped pattern; if $R^k$ is negatively oriented, they form a $N$-shape reflected over the $x$-axis (see Figure~\ref{fig:rajala1}). 

To prove the statement of the Lemma, we use induction on $k$.
\begin{figure}
    \centering
    \includegraphics[width=0.5\linewidth]{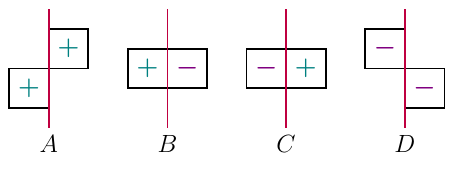}
    \caption{The four possible configurations $A,B,C,D$ in which the line $x=x_0$ (in red) intersects two consecutive rectangles of $G_k$, with their respective orientations.}
    \label{fig:tapio2}
\end{figure}
For $k=0$ we have $G_0=Q$ and the statement is trivial. For $k=1$, the statement is clearly true, see Figure~\ref{fig:rajala1}.
Next, we assume that the statement holds for some level $k\geq1$ and we prove it for $k+1$. We distinguish two cases. First, suppose that the line $x=x_0$ intersects $G_k$ in a single rectangle $R^k$. Since the collection of level-$(k+1)$ rectangles contained in $R^k$ forms the same rescaled $N$-shape (or reflected $N$-shape) inside $R_k$, the configurations of the intersection of $x=x_0$ with $G_{k+1}$ are exactly as in the case $k=1$. Thus, the statement holds.

The second and nontrivial case is when the line $x=x_0$ intersects $G_k$ in two consecutive rectangles, say $R_1^k=f_{i_1}^{j_1} \circ \cdots \circ f_{i_{k}}^{j_{k}}(Q)$ and  $R_2^k= f_{h_1}^{j_{1}} \circ \cdots \circ f_{h_{k}}^{j_{k}}(Q)$, ordered from left to right. Then, at level $k+1$ the line $x=x_0$ intersects only the rightmost level-$(k+1)$ rectangle nested inside $R_1^k$, only the leftmost level-$(k+1)$ rectangle nested in $R_2^k$. This two level-$(k+1)$ rectangles are, respectively, 
\begin{equation*}
    R_1^{k+1}= f_{i_1}^{j_1} \circ \cdots \circ f_{i_{k}}^{j_{k}}\circ f_{3m-1}^{j_{k+1}}(Q)\quad\text{and}\quad R_2^{k+1}\coloneqq f_{l_1}^{j_{1}} \circ \cdots \circ f_{l_{k}}^{j_{k}}\circ f_{0}^{j_{k+1}}(Q).
\end{equation*} 
By definition of orientation, these two rectangles have the same orientation of $R_1^k$ and $R_2^k$, respectively. 
Moreover, their configuration is the same as the $A,B,C$ or $D$ configuration of $R_1^k$ and $R_2^k$, as shown in Figure~\ref{fig:tapio4}.
\begin{figure}
    \centering
    \includegraphics[width=0.9\linewidth]{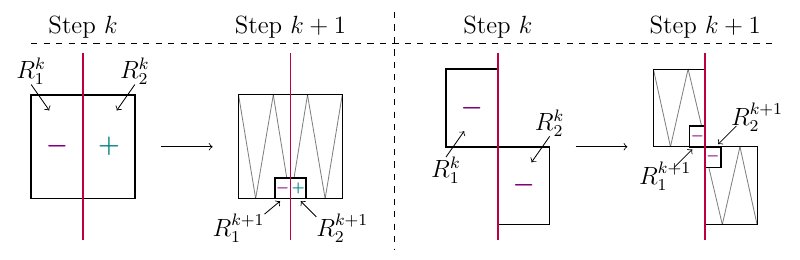}
    \caption{An illustration of how the line $x=x_0$ (in red) intersects two consecutive rectangles at steps $k$ and $k+1$, with configuration $C$ shown on the left and configuration $D$ on the right.}
    \label{fig:tapio4}
\end{figure}
\end{proof}

\subsection{Construction of measures on level sets}\label{subs: frostman}

To compute the dimension of the level sets of $\psi$, we use the Frostman Lemma, see \cite[Theorem 8.8]{Mat95}. 
\begin{lemma}\label{frostman lemma}
    Let $A\subset\R^n$ be a Borel set and let $0<s\le n$. Then the following are equivalent:
    \begin{enumerate}[label=(\roman*)]
        \item $\HH^s(A)>0$;
        \item there exists $C>0$ and a Radon measure $\mu$ compactly supported in $A$, with $0<\mu(\R^n)<+\infty$, and satisfying $\mu(B(x,r))\leq Cr^s$ for every $x\in\R^n$ and $r>0$.
    \end{enumerate}
\end{lemma}
Our objective is to show that $\HH^s(\psi^{-1}(t))>0$ for every $s< 1-\alpha$. By Frostman lemma, this reduces to constructing a Radon measure $\mu$ supported on $\psi^{-1}(t)$ with the appropriate scaling exponent. 
We start from the following technical lemma, whose proof can be carried out in the same spirit as that of Lemma~\ref{lem:Tapio graph}.
\begin{lemma}\label{lem: Gk intersect y=y_0}
   For every $k\in\N^+$ the set $G_k\cap\{y=t\}$ intersects $3^k$ or $2\cdot3^k$ rectangles of the level $k$, therefore forming $3^k$ intervals of length $b_k$ or $2b_k$ and having pairwise disjoint interiors. More precisely, each of these level-$k$ intervals contains $3$ level-$(k+1)$ sub-intervals. 
\end{lemma}
Let 
\begin{equation*}
    \{1,2,3\}^{<\N^+}\coloneqq \bigcup_{k\in\N^+}\{(\sigma_1,\ldots,\sigma_k)\colon \sigma_i\in\{1,2,3\} \text{ for each } i=1,\ldots,k\},
\end{equation*}
and 
\begin{equation*}
    \{1,2,3\}^{\N^+}\coloneqq \{(\sigma_1,\sigma_2,\sigma_3,\ldots)\colon \sigma_i\in\{1,2,3\} \text{ for each } i\in\N^+\},
\end{equation*}
be the spaces of finite and infinite sequences of elements of $\{1,2,3\}$, respectively. We index the $3^k$ intervals in $G_k\cap\{y=t\}$ given by Lemma \ref{lem: Gk intersect y=y_0} using elements of $\{1,2,3\}^{\N^+}$ as follows. We denote by $J_1$, $J_2$ and $J_3$ the intervals at the first level $G_1\cap\{y=0\}$. Then, if $J_{\sigma_1,\ldots,\sigma_k}$ denotes a level-$k$ interval, then at level $k+1$ it contains $3$ subintervals, which we enumerate from left to right as $J_{\sigma_1,\ldots,\sigma_k,1},J_{\sigma_1,\ldots,\sigma_k,2},J_{\sigma_1,\ldots,\sigma_k,3}$.

We aim to define a probability measure on $G\cap\{y=t\}$ by distributing mass equally among the $3^k$ intervals of each level $k$.

\emph{Step 1.} We construct a Borel measure $\nu$ on the space $X\coloneqq \{1,2,3\}^{\N^+}$ endowed with the product topology induced by the discrete topology on $\{1,2,3\}$.
Consider the base of clopen sets for $X$:
\begin{equation*}
    \mathcal{I}\coloneqq \{I_{\sigma_1,\ldots,\sigma_k}\colon k\in\N^+, \sigma_i=1,2,3\}\cup X,
\end{equation*}
where
\begin{equation*}
\begin{split}
I_{\sigma_1,\ldots,\sigma_k}\coloneqq \{\rho\in X\colon \rho_i=\sigma_i\text{ for all } i=1,\ldots k\}
    \end{split}
\end{equation*}
We define the function $\nu\colon \mathcal{I}\to\R$ such that
\begin{equation*}    \nu(X)=1\quad\text{and}\quad\nu(I_{\sigma_1,\ldots,\sigma_k})=3^{-k}.
\end{equation*}
It is easily observed that $\nu$ is a pre-measure on the semiring of sets  $\mathcal{I}$. 
Thus, by the Carath\'odory--Hahn theorem 
(see for instance \cite{RoydenFitzpatrick2010}), 
$\nu$ can be extended to a Borel regular outer measure on $X$.

\emph{Step 2.} We transport 
$\nu$ to a measure $\mu$ on $G\cap\{y=t\}$. We consider the continuous map
\begin{equation*}
\begin{split}
    \Psi\colon X&\to G\cap\{y=t\}\\
    \sigma &\mapsto x_\sigma\coloneqq \bigcap_{k=1}^\infty J_{\sigma_1,\ldots,\sigma_k}.
\end{split}    
\end{equation*}
Therefore the measure 
\begin{equation*}
    \mu\coloneqq \Psi_{\#}(\nu)
\end{equation*}
is Borel regular and finite on $\R^n$, thus it is a Radon measure. We observe that $\mu$ is supported on $G\cap\{y=t\}$ and $\mu(G\cap\{y=t\})=\nu(X)=1$. For every $(\sigma_1,\ldots,\sigma_k)\in\{1,2,3\}^k$ the following inclusions holds
\begin{equation*}
    \Psi^{-1}(\text{int}(J_{\sigma_1,\ldots,\sigma_k}))\subseteq I_{\sigma_1,\ldots,\sigma_k}\subseteq \Psi^{-1}(J_{\sigma_1,\ldots,\sigma_k}).
\end{equation*}
Moreover, $\mu(\{x_\sigma\})=0$ for every $x_\sigma\in G\cap \{y=t\}$, which follows from \cite[Lemma~9.2]{Falconer2003} by taking the sets $V_i$ of the lemma to be the level-$k$ intervals, for each $k\in\N$. As a consequence,
\begin{equation*}
    \mu(\text{int}(J_{\sigma_1,\ldots,\sigma_k}))=\nu( I_{\sigma_1,\ldots,\sigma_k})=\mu(J_{\sigma_1,\ldots,\sigma_k})=3^{-k}.
\end{equation*}
In the proof of Theorem \ref{thm: rajala holder+} we will show that
\begin{equation}\label{eq: frostman for Jk}    \mu(J_{\sigma_1,\ldots,\sigma_k})\leq C\diam(J_{\sigma_1,\ldots,\sigma_k})^s
\end{equation}
for every $k\in \N^+$ and $(\sigma_1,\ldots,\sigma_k)\in\{1,2,3\}^k$.
We observe that \eqref{eq: frostman for Jk} implies that $\mu$ satisfies condition (ii) of Lemma~\ref{frostman lemma}. Indeed, by \cite[Lemma~9.2]{Falconer2003}, an arbitrary interval $J\subset\{y=t\}$ such that $b_{k+1}<\diam(J)\leq b_k$ intersects a bounded number $M_0$ of $k$-level intervals $J_{\sigma_1,\ldots,\sigma_k}$. Recalling that $b_{k}\leq3(m+1)b_{k+1}$ and using \eqref{eq: frostman for Jk} we obtain
\begin{equation*}
\begin{split}
\mu(J)&\leq M_0\mu(J_{\sigma_1,\ldots,\sigma_k})\leq CM_0 (b_k)^s\\ 
&\leq CM_03^s(m+1)^s (b_{k+1})^s\leq  CM_03^s(m+1)^s \diam(J)^s.
\end{split}
\end{equation*}

We now prove Theorem \ref{thm: rajala holder+}. The computations below rely on arguments similar to the proof of \cite[Proposition~2.4]{Rajala2026pr}, except for the choice of the function $\varphi(k)=\lfloor \rho k\rfloor$, with $\rho$ as in \eqref{eq: rho_alpha}.

\begin{proof}[Proof of Theorem~\ref{thm: rajala holder+}]
Fix $\alpha,\beta\in(0,1)$ and let $m\in\N^+$ such that
    \begin{equation*}
        \frac{\log m}{\log (3m)}\leq \alpha < \frac{\log (m+1)}{\log (3(m+1))}.
    \end{equation*}
    We define the function
    \begin{equation*}
        \varphi(k)=\lfloor\rho k + (1-\rho)k^{\beta^*}\rfloor,
    \end{equation*}
    where $\beta^*\in (\beta,1)$ and
    \begin{equation}\label{eq: rho_alpha} 
    \rho\coloneqq \frac{\alpha\log3-(1-\alpha)\log m}{(1-\alpha)(\log(m+1)-\log m)}\in[0,1).
    \end{equation}
    We observe that $\varphi$ is nondecreasing and satisfies $\varphi(k)\leq\varphi(k-1)+1$ for $k\in\N^+$. We argue as in subsection~\ref{subs: construction of g}, and denote the resulting function by $\psi$. First, we prove (ii). We fix $x,y\in[0,1]$ and denote by $k\in\N$ the unique integer such that
    \begin{equation*}
        b_{k+1}<|x-y|\leq b_k.
    \end{equation*}
    Then, for every $M>0$ we have
    \begin{equation*}
        e^{M|\log|x-y||^\beta}\leq e^{M\big(\!\log\left((3(m+1)
    )^{\varphi(k+1)}(3m)^{k+
        1-\varphi(k+1)}\right)\big)^\beta}\leq C e^{CMk^\beta}, 
    \end{equation*}
    where here and in the following $C>0$ denotes a constant depending only on $\alpha$ and $\beta$, whose value may change from line to line.
    On the other hand, since $|x-y|<b_k$, $\psi(x)$ and $\psi(y)$ belong to the same level-$k$ rectangle or to two adjacent level-$k$ rectangles. Therefore we have
    \begin{equation*}
    \begin{split}
        \frac{|\psi(x)-\psi(y)|}{|x-y|^\alpha}&\leq \frac{2h_k}{(b_{k+1})^\alpha}= 2\frac{\left(3(m+1))^{\varphi(k+1)}(3m)^{k+1-\varphi(k+1)}\right)^\alpha}{(m+1)^{\varphi(k)}m^{k-\varphi(k)}}\\
    &=2 e^{[\alpha(k+1-\varphi(k+1))-k+\varphi(k)]\log m+\alpha(k+1)\log 3+[\alpha\varphi(k+1)-\varphi(k)]\log (m+1)}\\
    &\leq C e^{(\alpha-1)(k-\varphi(k))\log m+\alpha k\log 3+(\alpha-1)\varphi(k)\log (m+1)}.
        \end{split}
    \end{equation*}
    Substituting the expression of $\varphi$ and exploiting the cancellations arising from the specific choice of $\rho$, we obtain
    \begin{equation*}
\begin{split}
     \frac{|\psi(x)-\psi(y)|}{|x-y|^\alpha}e^{M|\log|x-y||^\beta}\leq
        Ce^{-(1-\alpha)(1-\rho)(\log(m+1)-\log m)k^{\beta^*}+CMk^\beta}.
    \end{split}
\end{equation*}
    The right-hand side of the previous equation vanishes as $k\to\infty$, therefore it is equi-bounded in $k$. This proves (ii).

    To estimate the dimension of the level sets of $\psi$, we argue as in subsection~\ref{subs: frostman}. For every $\eps>0$, it suffices to verify \eqref{eq: frostman for Jk} with $s=1-\alpha-\eps$:
    \begin{equation*}
\begin{split}
\frac{\mu(J_{\sigma_1,\ldots,\sigma_k})}{\diam(J_{\sigma_1,\ldots,\sigma_k})^{1-\alpha-\eps}}&\leq 2 \frac{\left((3(m+1))^{\varphi(k)}(3m)^{k-\varphi(k)}\right)^{1-\alpha-\eps}}{3^{k}}\\
    &=2 e^{-k(\alpha+\eps)\log 3+\varphi(k)(1-\alpha-\eps)(\log(m+1)-\log m)+(1-\alpha-\eps)k\log m}\\
    &\leq 2e^{-k\alpha\log 3+\rho k(1-\alpha)(\log(m+1)-\log m)+(1-\alpha)k\log m}\\
    &\hspace{0.4cm}\cdot e^{-\eps k[\log3+ \log m+\rho(\log(m+1)-\log m)]+k^{\beta^*}(1-\rho)(1-\alpha-\eps)(\log(m+1)-\log m)}.
\end{split}
\end{equation*}
In the last inequality, the exponent of the first exponential vanishes identically by the definition of $\rho$, while the second exponential goes to zero as $k\to\infty$ since $\beta^*<1$. Thus \eqref{eq: frostman for Jk} holds and therefore property (i) follows.
\end{proof}

\subsection{Further results on level sets of H\"older functions}

To conclude the section, we provide a few more observations on level sets of H\"older continuous functions. 
The next remark, for $n=1$, shows that the functions constructed in \cite[Proposition~2.6]{Rajala2026pr} and in Theorem~\ref{thm: rajala holder+} have a.e.\ level set of maximal Hausdorff dimension equal to $1-\alpha$. Moreover, Lemma~\ref{lem: holder+ level set} shows that a.e.\ level set of the functions in Theorem~\ref{thm: rajala holder+} cannot have positive $(1-\alpha)$-dimensional Hausdorff measure.

\begin{remark}\label{rem: dim holder level set}
We recall that if $f\colon \Omega\to\R$ is an $\alpha$-H\"older continuous function, with $\Omega\subset\R^n$ of finite Lebesgue measure and $0<\alpha\leq 1$, then $\HH^{n-\alpha}(f^{-1}(t))<+\infty$ for a.e.\ $t\in\R$, see also \cite[Corollary 11.2]{Falconer2003} or \cite[Remark 3.8]{KirSer04}. This can be obtained by a proper use of
Eilenberg's inequality (Theorem~\ref{lemma:Eilenberg}), 
arguing as in the proof of the next lemma.
\end{remark}

\begin{lemma}\label{lem: holder+ level set}
If $0<\alpha\le1$, $\Omega\subset\R^n$ is open, and $\psi\colon \Omega\to\R$ satisfies
       \begin{equation}\label{eq: psi is h-holder}
            \lim_{r\to0}\sup_{\substack{0<|x-y|\leq r \\ x,y\in\Omega'}}\frac{|\psi(x)-\psi(y)|}{|x-y|^\alpha}=0,
        \end{equation}   
for every open subset $\Omega'\Subset\Omega$, then $\HH^{n-\alpha}(\psi^{-1}(t))=0$ 
for a.e.\ $t\in\R$.
\end{lemma}
    \begin{proof}
We choose a nonempty open set $\Omega'\Subset\Omega$ and fix $\eps>0$.
Then we can find a finite cover 
$\{B_i:\,1\le i\le N\}$ of $\overline{\Omega'}$, where $B_i$ are open Euclidean balls of suitably small radius, such that the Lipschitz constant
of $\psi|_{B_i\cap\Omega'}$ with respect to the metric $\rho(x,y)=|x-y|^\alpha$
is less than $\eps$.
For each $\beta>0$, we denote by $\widetilde\HH^\beta$ the Hausdorff measure with respect to $\rho$ and observe that $\widetilde \HH^\beta=\frac{\omega_\beta}{\omega_{\beta\alpha}}\HH^{\beta\alpha}$.
We can easily find a finite measurable partition $\{E_j:\,1\le j\le N_1\}$
of $\Omega'$, starting from the measurable subsets $B_i\cap\Omega'$, such that all the Lipschitz constants of $\psi|_{E_i}$ are less than $\eps$.
By Theorem~\ref{lemma:Eilenberg} with unitary gauges, it follows that
    \begin{equation*}
     \begin{split}
     &\int_\R \widetilde\HH^{\frac{n}{\alpha}-1}(\psi^{-1}(t)\cap\Omega')\,dt)
     =\sum_{i=1}^{N_1}\int_\R \widetilde\HH^{\frac{n}{\alpha}-1}(\psi^{-1}(t)\cap E_i)dt\\
     &\leq 
     \frac{\omega_{\frac{n}{\alpha}-1} \, \omega_1}{\omega_\frac{n}{\alpha}}
     \sum_{i=1}^N\eps\, \widetilde\HH^{\frac{n}{\alpha}}(E_i)\leq
     \frac{\omega_{\frac{n}{\alpha}-1} \, \omega_1}{\omega_\frac{n}{\alpha}}
     \eps\,\widetilde\HH^{\frac{n}{\alpha}}(\Omega')=\frac{\omega_{\frac{n}{\alpha}-1} \, \omega_1}{\omega_n}\eps\,\HH^n(\Omega').
        \end{split}
        \end{equation*}
The arbitrary choice of $\eps>0$ concludes the proof.
\end{proof}

\bibliographystyle{abbrv}
\bibliography{References}
\end{document}